\documentclass[11pt]{article}

\usepackage{amsfonts}
\usepackage{amscd}
\usepackage{amssymb}
\usepackage{amsthm}
\usepackage{amsmath}
\usepackage{stmaryrd}
\usepackage{graphicx}
\usepackage{color}
\usepackage{verbatim}
\usepackage{mathrsfs}
\usepackage{dsfont}
\usepackage{tikz}
\usepackage{subfigure}
\usetikzlibrary{arrows,decorations.pathmorphing,backgrounds,positioning,fit}

\colorlet{Green}{black!20!green}

 \theoremstyle{plain}
\newtheorem{thm}{Theorem}[section]

\newtheorem{lemma}[thm]{Lemma}
\newtheorem{Lemma}[thm]{Lemma}
\newtheorem{lem}[thm]{Lemma}
\newtheorem{prop}[thm]{Proposition}
\newtheorem{Prop}[thm]{Proposition}
\newtheorem{cor}[thm]{Corollary}
\theoremstyle{definition}

\newtheorem{defn}[thm]{Definition}
\newtheorem{remark}[thm]{Remark}
\theoremstyle{remark}

\numberwithin{equation}{section}

\def\cA{\mathcal{A}}
\def\cB{\mathcal{B}}

\def\cR{\mathcal{R}}

\def\NN{\mathbb{N}}

\def\PP{\mathbb{P}}

\def\RR{\mathbb{R}}

\def\ZZ{\mathbb{Z}}

\newcommand{\g}{\gamma}
\newcommand{\E}{\mathbb{E}}
\newcommand{\N}{\mathbb{N}}
\newcommand{\R}{\mathbb{R}}
\newcommand{\Prob}{\mathbb{P}}
\newcommand{\bT}{\mathbf{T}}
\newcommand{\T}{\mathcal{T}}

\newcommand{\W}{\mathbf{W}}
\newcommand{\Y}{\mathbf{Y}}

\newcommand{\A}{\mathcal{A}}

\newcommand{\Paths}{\mathcal{P}}
\newcommand{\mfA}{\mathfrak{A}}

\newcommand{\x}{\mathbf{x}}

\newcommand{\e}[1]{\mathbf{e}_{#1}}
\newcommand{\w}{\mathfrak{w}}

\colorlet{lightgray}{black!20}

\makeatletter
\renewcommand{\@makefnmark}{\mbox{\textsuperscript{}}}
\makeatother

\tikzset{
  pics/carc/.style args={#1:#2:#3}{
    code={
      \draw[pic actions] (#1:#3) arc(#1:#2:#3);
    }
  }
}
\title{Asymptotic entropy of random walks on Fuchsian buildings and Kac-Moody groups}
\author{
Lorenz Gilch\footnote{Research supported by the exchange programme OeAD Amadeus-Campus France FR $11/2014$ 
} \and Sebastian M\"{u}ller\footnote{Research partly supported under CIRM Research in Pairs program and PHC Amadeus $31473$TF} \and James Parkinson\footnote{Research partly supported under the Austrian Science Fund W1230 and P24028, and CIRM Research in Pairs program}
}

\begin{document}

\maketitle

\begin{abstract}
In this article we prove existence of the asymptotic entropy for isotropic random walks on regular Fuchsian buildings. Moreover, we give formulae for the asymptotic entropy, and prove that it is equal to the rate of escape of the random walk with respect to the Green distance. When the building arises from a Fuchsian Kac-Moody group our results imply results for random walks induced by bi-invariant measures on these groups, however our results are proven in the general setting without the assumption of any group acting on the building. The main idea is to consider the retraction of the isotropic random walk onto an apartment of the building, to prove existence of the asymptotic entropy for this retracted walk, and to `lift' this in order to deduce the existence of the entropy for the random walk on the building.
\end{abstract}

\section{Introduction}

\textit{Buildings} are combinatorial/geometric objects which arose from the fundamental work of Jacques Tits related to semisimple Lie groups and algebraic groups. Celebrated results in the theory include Tits' classification theorems for irreducible thick \textit{spherical} buildings of rank at least~$3$, and irreducible thick \textit{affine} buildings of rank at least~$4$ (see Tits \cite{Ti:74} and \cite{Ti:86}), showing that these classes of buildings are essentially equivalent to certain classes of Lie-type groups. Later developments in the theory of more general \textit{Kac-Moody groups} stimulated interest in buildings of more general type (see Tits \cite{Ti:87}). 

The theory of random walks on buildings has, until recently, focused primarily on the spherical and affine cases (see the survey of Parkinson \cite{Par:15}). In \cite{gilch-mueller-parkinson:14} we initiated an investigation of probability theory on \textit{Fuchsian} buildings, and proved a rate of escape theorem and a central limit theorem for random walks on these buildings. In this sequel to~\cite{gilch-mueller-parkinson:14} we study the \textit{asymptotic entropy} for random walks on Fuchsian buildings and associated groups, which is another important characteristical number of random walks.

Before stating our main results we give some background. Let
$(W,S)$ be a Coxeter system. A \textit{building $(\Delta,\delta)$ of type
  $(W,S)$} consists of a set $\Delta$ (whose elements are the \textit{chambers}
of the building) along with a ``generalised distance function''
$\delta:\Delta\times\Delta\to W$ satisfying various axioms analogously to the usual metric space axioms (see
Definition~\ref{defn:building}). Thus the ``distance'' between chambers
$x,y\in\Delta$ is an element $\delta(x,y)$ of the Coxeter group~$W$, and by taking word length in~$W$ this
gives rise to a metric $d(\cdot,\cdot)$ on the building. 

The structure of the building $(\Delta,\delta)$ is heavily influenced by the Coxeter system $(W,S)$. If $W$ is  a finite reflection group then the building is called \textit{spherical}, and if $W$ is a Euclidean reflection group then the building is called \textit{affine}. In this paper we are interested in the case where $W$ is a group generated by reflections in the hyperbolic plane, in which case the building is called a \textit{Fuchsian building}. In general, buildings contain many copies of the Cayley graph of the Coxeter system $(W,S)$. These substructures are called the \textit{apartments} of the building, and in the case of a Fuchsian building each apartment can be realised as a tessellation of the hyperbolic disc $\mathbb{H}^2$ by the reflection group~$W$. 

While the classification theorems for spherical and affine buildings imply that these buildings always admit transitive automorphism groups if the rank of the building is sufficiently large, the case for Fuchsian buildings is quite different and there exist free constructions of these buildings with the only constraints being the existence of certain local geometries (see Ronan \cite{Ron:86} and \mbox{\cite[\S 2.4]{gilch-mueller-parkinson:14}).} Thus Fuchsian buildings typically admit no transitive group action.

We consider the class of \textit{isotropic random walks} $(X_n)_{n\geq 0}$ on the set $\Delta$ of chambers of a Fuchsian building. These random walks have the property that their single-step transition probabilities $p(x,y)$, $x,y\in\Delta$, depend only on the generalised distance~$\delta(x,y)$. While we proved  in \cite{gilch-mueller-parkinson:14} a rate of escape theorem for isotropic random walks on Fuchsian buildings and associated groups, we turn our attention in the present paper to an investigation of the \textit{asymptotic entropy} for isotropic random walks on Fuchsian buildings.

The \textit{(asymptotic) entropy}, as introduced by Avez \cite{avez72}, of a random walk $(X_n)_{n\geq 0}$ is defined by
$$
h=\lim_{n\to\infty}-\frac{1}{n}\mathbb{E}[\log\pi_n(X_n)]\quad\text{if the limit exists},
$$
where $\pi_n$ is the distribution of $X_n$. Intuitively, the entropy can be seen as the ``asymptotic uncertainty'' of the random walk and it plays an important role in the description of the asymptotic behavior of the random walk, see Derriennic [9, 10], Guivarc?h [15], Kaimanovich [17], Kaimanovich and Vershik [18] and Vershik [24], amongst others.

It is well-known that the limit defining $h$ necessarily exists for random walks on groups whenever $\mathbb{E}[-\log \pi_1(X_1)]<\infty$ (this is an application of Kingman's subadditive ergodic theorem~\cite{kingman}). However existence of the entropy for more general structures is not known a priori. In particular in our setting of Fuchsian buildings there is typically no transitive group action on the building, and hence the existence of entropy cannot be straightforwardly deduced from Kingman's theorem due to lack of subadditivity. This forces us to employ other techniques, primarily \textit{generating functions}. 

Our generating function approach is motivated by the work of Benjamini and Peres
\cite{benjamini-peres94} where it is shown that for random walks on groups the
entropy equals the rate of escape with respect to the Green distance; compare also
with Blach\`ere, Ha\"issinsky and Mathieu \cite{blachere-haissinsky-mathieu}. Indeed we prove in this paper that the entropy $h$ exists for isotropic random walks on Fuchsian buildings, and show that $h$ equals the rate of escape with respect to the Green distance. Moreover, we prove that the sequence $-\frac{1}{n} \log \pi_n(X_n)$ converges to $h$ in $L_1$, and we also show that $h$ can be computed along almost every sample path as the limit inferior of the aforementioned sequence. The question of almost sure convergence of $-\frac{1}{n}\log\pi_n(X_n)$ to $h$, however, remains open. 

The main idea in our proofs is to project the random walk $(X_n)_{n\geq 0}$ on the building onto the Coxeter group $W$ by setting $\overline{X}_n=\delta(o,X_n)$, where $o\in\Delta$ is some fixed chamber (the \textit{origin} of the building). Geometrically this projection is a \textit{retraction} of the building onto an apartment, and we refer to $(\overline{X}_n)_{n\geq 0}$ as the \textit{retracted walk}. It turns out that $(\overline{X}_n)_{n\geq 0}$ is indeed a random walk on $W$, and we write $\overline{p}(u,v)$, $u,v\in W$, for the transition probabilities for this walk. However we note that the retracted walk is not $W$-invariant, that is, $\overline{p}(wu,wv)\neq\overline{p}(u,v)$ in general for $w\in W$. In order to track the projected random walk's path to infinity we construct a sequence of nested cones in $(W,S)$ which allows us to cut the random walk trajectory into aligned pieces such that these pieces arise as the realization of some Markov chain. From this construction we are able to deduce the existence of the entropy of the retracted walk, which is given by $\overline{h}=\lim_{n\to\infty} -\frac1n \mathbb{E}\bigl[\log \overline{\pi}(\overline{X}_n)\bigr]$ with $\overline{\pi}$ being the distribution of $\overline{X}_n$. It is then relatively straightforward to prove the existence of the entropy for the random walk on the building. 
Denote by $p^{(k)}(x,y)$ the $k$-step transition probabilities of $(X_n)_{n\geq 0}$ and by $\overline{p}^{(k)}(u,v)$ the $k$-step transition probabilities of $(\overline{X}_n)_{n\geq 0}$. 
 We define the Green functions
\begin{equation*}
G(x,y):=\sum_{k=0}^{\infty} p^{(k)}(x,y)\ \mbox{ and }\ \overline{G}(u,v):=\sum_{k=0}^{\infty}  \overline{p}^{(k)}(u,v).
\end{equation*}

With the above notations and definitions in hand, the main results of this paper are summarised in the following.

\begin{thm}\label{thm:main}
Let $(\Delta,\delta)$ be a locally finite thick regular Fuchsian building of type~$(W,S)$. 
Let $(X_n)_{n\geq 0}$ be an isotropic random walk on $\Delta$ with bounded range, and let $(\overline{X}_n)_{n\geq0}$ be the associated retracted random walk on $W$. Then:
\begin{enumerate}
\item The asymptotic entropy $\overline{h}$ of $(\overline{X}_n)_{n\geq 0}$ exists, and equals the rate of escape of $(\overline{X}_n)_{n\geq 0}$ with respect to the Green distance. That is,
$$
\overline{h}=\lim_{n\to\infty}-\frac{1}{n} \log \overline{G}(e,\overline{X}_n)\quad\text{almost surely}.
$$
\item The asymptotic entropy $h$ of $(X_n)_{n\geq 0}$ exists, and is equal to the rate of escape of $(X_n)_{n\geq 0}$ with respect to the Green distance. That is,
$$
h=\lim_{n\to\infty}-\frac{1}{n}\log {G}(e,{X}_n)\quad\text{almost surely.}
$$
Moreover, we have the formula
$$
h=\overline{h}+h_q,\quad\text{where}\quad
h_q=\lim_{n\to\infty}\frac1n \log q_{\overline{X}_n} \quad\text{almost surely},
$$
where for $w\in W$ the number $q_w$ is the cardinality of the sphere $\{x\in\Delta\mid \delta(o,x)=w\}$. 
\end{enumerate}
\end{thm}
We note that Theorem \ref{thm:L-limit} together with (\ref{equ:h_G}) gives a formula for $\overline{h}$ in terms of the rate of escape and the entropy of a hidden Markov chain, while Proposition \ref{prop:hq-convergence} gives a formula for $h_q$. 

We also note that Ledrappier and Lim~\cite{LL:10} have investigated \textit{volume entropy} for hyperbolic buildings (that is, the exponential growth rate of balls in the building). While this form of entropy is quite different to the asymptotic entropy considered here, it is interesting to note the similar forms to the formulae in Theorem~\ref{thm:main} and \cite[Theorem~1.1]{LL:10}. 

The framework of our proofs follows ideas from Gilch \cite{gilch:14}, where the entropy for random walks on regular languages is investigated. Similar results, in different contexts, concerning existence of the entropy are proved in Gilch and M\"uller \cite{gilch-mueller} for random walks on directed covers of graphs, and in Gilch \cite{gilch:11} for random walks on free products of graphs. Moreover, a survey article on rate of escape and entropy of random walks is Gilch and Ledrappier~\cite{gilch-ledrappier13} and, in particular, for random walks on hyperbolic groups see Ledrappier \cite{ledrappier12-2}. 

When a group acts sufficiently transitively on a regular Fuchsian building Theorem~\ref{thm:main} provides information on the entropy for random walks associated to these groups. For examples of this conversion we refer to \cite[Corollary~1.3]{gilch-mueller-parkinson:14}.

In Section~\ref{sec:coxeter-groups} we give a brief introduction to Coxeter groups, Fuchsian buildings, and isotropic random walks on these structures (with more details available in the companion paper~\cite{gilch-mueller-parkinson:14}). In Section~\ref{sec:entropy-retracted-walk} we prove existence of the
asymptotic entropy of the retracted walk, and derive formulae for the entropy. In Section \ref{sec:entropy-building} we lift these results to the isotropic random walk on the building, completing the proof of Theorem~\ref{thm:main}.

\section{Coxeter groups and buildings}
\label{sec:coxeter-groups}

In this section we give an introduction to Fuchsian buildings and random walks on them:
we give the formal definition of Coxeter systems and Fuchsian buildings, define
isotropic random walks on them, and collect some important properties; we refer to \cite{gilch-mueller-parkinson:14} for more details.

\subsection{Fuchsian Coxeter groups}

\subsubsection{Coxeter systems}

A \textit{Coxeter system} $(W,S)$ is a group $W$ with neutral element $e$ generated by a finite set~$S\not\ni e$ with relations
$$
s^2=e\quad\textrm{and}\quad (st)^{m_{st}}=e\quad\textrm{for all $s,t\in S$ with $s\neq t$},
$$
where $m_{st}=m_{ts}\in\ZZ_{\geq 2}\cup\{\infty\}$ for all $s\neq t$ (if
$m_{st}=\infty$ then it is understood that there is no relation between $s$ and
$t$).  The \textit{word length} of $w\in W$ is defined as
$$
\ell(w):=\min\{n\geq 0\mid w=s_1\cdots s_n\textrm{ with }s_1,\ldots,s_n\in S\}
$$
and an expression $w=s_1\cdots s_n$ with $n=\ell(w)$ is called a
\textit{reduced expression} for~$w$. 
If $w\in W$ and $s\in S$ then $\ell(ws)\in\{\ell(w)-1,\ell(w)+1\}$. In
particular, $\ell(ws)=\ell(w)$ is impossible. The \textit{distance} between
elements $u\in W$ and $v\in W$ is defined as
$$
d(u,v):=\ell(u^{-1}v).
$$
The \textit{ball} of radius $n\in\N_0$ with centre $u\in W$ is
$\cB_n(u):=\{v\in W\mid d(u,v)\leq n\}$.
A \textit{path} in $W$ is a sequence of words $[u_0,u_1,\dots,u_n]$ such that
$u_{i-1}^{-1}u_i\in S$ for all $i\in\{1,\dots,n\}$. A path from $u$ to $v$ is a \textit{geodesic}
if it is a path of shortest length from $u$ to $v$. A \textit{ray} is an
infinite path of the form $[u_0,u_1,\dots ]$ such that $u_{i-1}^{-1}u_i\in S$
for all $i\in\N$. An \textit{infinite geodesic} is a ray starting at $e$ with
$d(e,u_i)=i$. Sometimes we will also identify paths or rays $\gamma$ by their
vertex sets $\{u_0,u_1,\dots\}$.

A Coxeter system $(W,S)$ is \textit{irreducible} if there is no partition of the generating set $S$ into disjoint nonempty sets $S_1$ and $S_2$ such that $s_1s_2=s_2s_1$ for all $s_1\in S_1$ and all $s_2\in S_2$. We will always assume that $(W,S)$ is irreducible. 

\subsubsection{Fuchsian Coxeter groups}

We now define a special class of Coxeter groups that are discrete subgroups of~$PGL_2(\mathbb{R})$, called \textit{Fuchsian Coxeter groups}. 
Let $n\geq 3$ be an integer, and let $k_1,\ldots,k_n\geq 2$ be integers satisfying 
\begin{align}\label{eq:hyperbolic}
\sum_{i=1}^n\frac{1}{k_i}<n-2.
\end{align}
Assign the angles $\pi/k_i$ to the vertices of a combinatorial $n$-gon~$F$. There is a convex realisation of $F$ (which we also call~$F$) in the hyperbolic disc~$\mathbb{H}^2$, and the subgroup of $PGL_2(\mathbb{R})$ generated by the reflections in the sides of~$F$ is a Coxeter group $(W,S)$ (see Davis \cite[Example~6.5.3]{davis}). If $s_1,\ldots,s_{n}$ are the reflections in the sides of~$F$ (arranged cyclically), then the order $m_{s_is_j}=m_{ij}$ of $s_is_j$ is 
\begin{align}\label{eq:hyperbolic2}
\begin{aligned}
m_{ij}=\begin{cases}k_i&\textrm{if $j=i+1$}\\
\infty&\textrm{if $|i-j|>1$},
\end{cases}
\end{aligned}
\end{align}
where the indices are read cyclically with $n+1\equiv 1$. 

A Coxeter system~$(W,S)$ given by data~(\ref{eq:hyperbolic})
and~(\ref{eq:hyperbolic2}) is called a \textit{Fuchsian Coxeter
  system}. Observe that these
systems are always infinite. The group $W$ acts on $\mathbb{H}^2$ with fundamental domain~$F$. Note that this action does not preserve orientation, however the index~$2$ subgroup~$W'$ generated by the even length elements of~$W$ is orientation preserving. Thus $W'$ is a discrete subgroup of $PSL_2(\mathbb{R})$, and so is a `Fuchsian group' in the strictest sense of the expression.

The Fuchsian Coxeter system $(W,S)$ induces a tessellation of $\mathbb{H}^2$ by isometric polygons $wF$, $w\in W$. The polygons $wF$ are called \textit{chambers}, and we usually identify the set of chambers with~$W$ by $wF\leftrightarrow w$. We call this the \textit{hyperbolic realisation} of the Coxeter system~$(W,S)$ (it is closely related to the \textit{Davis complex} from~\cite{davis}, see the discussion in Abramenko and Brown \cite[Example~12.43]{AB}). 

We refer e.g. to \cite[Example 2.1]{gilch-mueller-parkinson:14} for examples of
Fuchsian Coxeter systems.

\subsection{Fuchsian buildings}

We now give an axiomatic definition of buildings, following~\cite{AB}.

\begin{defn}\label{defn:building}
Let $(W,S)$ be an irreducible Fuchsian Coxeter system. A \textit{Fuchsian building of type $(W,S)$} is a pair $(\Delta,\delta)$ where~$\Delta$ is a nonempty set (whose elements are called \textit{chambers}) and $\delta:\Delta\times\Delta\to W$ is a function (called the \textit{Weyl distance function}) such that if $x,y\in\Delta$ then the following conditions hold:
\begin{enumerate}
\item[(B1)] $\delta(x,y)=e$ if and only if $x=y$.
\item[(B2)] If $\delta(x,y)=w$ and $z\in\Delta$ satisfies $\delta(y,z)=s$ with $s\in S$, then $\delta(x,z)\in\{w,ws\}$. If, in addition, $\ell(ws)=\ell(w)+1$, then $\delta(x,z)=ws$.
\item[(B3)] If $\delta(x,y)=w$ and $s\in S$, then there is a chamber $z\in\Delta$ with $\delta(y,z)=s$ and $\delta(x,z)=ws$. 
\end{enumerate}
\end{defn}

Let $(\Delta,\delta)$ be a building of type $(W,S)$ and let $s\in S$. Chambers
$x,y\in\Delta$ are \textit{$s$-adjacent} (written $x\sim_s y$) if
$\delta(x,y)=s$. One useful way to visualise a building is to imagine an
$|S|$-gon with edges labelled by the generators $s\in S$ (think of the edges as
being coloured by $|S|$ different colours). Call this $|S|$-gon the
\textit{base chamber} which we denote by $o$. Now take one copy of the base chamber for each element $x\in\Delta$, and glue these chambers together along edges so that $x\sim_s y$ if and only if the chambers are glued together along their $s$-edges.

A \textit{gallery of type $(s_1,\ldots,s_n)$} joining $x\in\Delta$ to $y\in\Delta$ is a sequence $x_0,x_1,\ldots,x_n$ of chambers with
$$
x=x_0\sim_{s_1}x_1\sim_{s_2}\cdots\sim_{s_n}x_n=y. 
$$
This gallery has \textit{length $n$}. 

The Weyl distance function $\delta$ has a useful description in terms of minimal length galleries in the building: if $s_1\cdots s_n$ is a reduced expression in~$W$ then $\delta(x,y)=s_1\cdots s_n$ if and only if there is a minimal length gallery in $\Delta$ from $x$ to $y$ of type $(s_1,\ldots,s_n)$. The \textit{(numerical) distance} between chambers $x,y\in\Delta$ is
$$
d(x,y):=\textrm{``length of a minimal length gallery joining $x$ to $y$''}=\ell(\delta(x,y)),
$$
Note that we use the same notation $d(\cdot,\cdot)$ for distance in both the Coxeter system and the building.

A building $(\Delta,\delta)$ is called \textit{thick} if $|\{y\in\Delta\mid x\sim_s y\}|\geq 2$ for all chambers $x\in\Delta$ and all $s\in S$, and \textit{thin} if $|\{y\in\Delta\mid x\sim_s y\}|=1$ for all chambers $x\in\Delta$ and all $s\in S$. For each Coxeter system $(W,S)$ there is (up to isomorphism) a unique thin building of type $(W,S)$. This thin building is called the \textit{Coxeter complex} of $(W,S)$, and has $\Delta=W$ and $\delta=\delta_W$ where $\delta_W(u,v)=u^{-1}v$ for all $u,v\in W$.

A building $(\Delta,\delta)$ is \textit{regular} if 
$$
q_s:=|\{y\in\Delta\mid x\sim_s y\}|\quad\text{is finite and does not depend on $x\in\Delta$}.
$$
For the remainder of this paper we will assume that $(\Delta,\delta)$ is regular. The numbers $(q_s)_{s\in S}$ are called the \textit{thickness parameters} of the building.

For each $x\in\Delta$ and each $w\in W$, let
$$
\Delta_w(x):=\{y\in\Delta\mid \delta(x,y)=w\}\quad\text{be the ``sphere of radius $w$'' centred at $x$.}
$$
 By Parkinson \cite[Proposition~2.1]{P1} the cardinality $q_w=|\Delta_w(x)|$ does not depend on $x\in \Delta$, and is given by
$$
q_w=q_{s_1}\cdots q_{s_{\ell}}\quad\textrm{whenever $w=s_1\cdots s_{\ell}$ is a reduced expression.}
$$

If $(\Delta',\delta')$ is a building with $\Delta'\subseteq \Delta$ and
$\delta|_{\Delta'}=\delta'$ then we call $(\Delta',\delta')$ a
\textit{sub-building} of $(\Delta,\delta)$.

Finally, let us remark that we will typically use the letters $u,v,w$ for elements of a Coxeter group~$W$, and the letters $x,y,z$ for chambers of a building~$(\Delta,\delta)$. 

\subsection{Apartments and retractions}\label{sect:aptret}

Let $(W,S)$ be an irreducible Fuchsian Coxeter system, and let $(\Delta,\delta)$ be a building of type $(W,S)$. The thin sub-buildings of
$(\Delta,\delta)$ of type $(W,S)$ are called the \textit{apartments} of $(\Delta,\delta)$. Thus each apartment is isomorphic to the Coxeter complex of~$(W,S)$. Two key facts concerning apartments are as follows:
\begin{enumerate}
\item[(A1)] If $x,y\in\Delta$ then there is an apartment $A$ containing both $x$ and $y$.
\item[(A2)] If $A$ and $A'$ are apartments containing a common chamber $x$ then there is a unique isomorphism $\theta:A'\to A$ fixing each chamber of the intersection $A\cap A'$.
\end{enumerate}
In fact conditions~(A1) and (A2) can be taken as an alternative, equivalent definition of buildings (see \cite[Definition~4.1]{AB} for the precise statement, and \cite[Theorem~5.91]{AB} for the equivalence of the two axiomatic systems).

Roughly speaking, the properties (A1) and (A2) ensure that the hyperbolic metric on each apartment can be coherently `glued together' to make $(\Delta,\delta)$ a $\mathrm{CAT}(-1)$ space (see \cite[Theorem~18.3.9]{davis} for details). 

Fix, once and for all, an apartment $A_0$ and a chamber $o\in A_0$. Identify $A_0$ with the Coxeter complex of $(W,S)$ (or just with $W$) such that $o$ is identified
with $e$, the neutral element of~$W$. Thus we regard $W=A_0$ as a ``base
apartment'' of~$\Delta$. The \textit{retraction $\rho$ of $\Delta$ onto $W$
  with centre~$o$} is the function $\rho:\Delta\to W$ with
\begin{align}\label{eq:canonicalretraction}
\rho(x):=\delta(o,x).
\end{align}
Alternatively, let $A'$ be any apartment containing $o$ and $x$ (using (A1)) and let $\theta:A'\to W$ be the isomorphism from (A2) fixing $W\cap A'$. Then
$
\rho(x)=\theta(x).
$
Thus $\rho$ ``radially flattens'' the building onto $W$, with centre~$o\in A_0=W$. 

We also note that in the apartment $A_0=W$ the Weyl distance function is given by
$$
\delta(u,v)=u^{-1}v\quad\text{for all $u,v$ in the base apartment~$W$}.
$$

\subsection{Automata for Coxeter Groups}

The notions of \textit{cones}, \textit{cone types}, and \textit{automata} are well established for finitely generated groups, with Epstein, Cannon et al.\cite{cannon} being a standard reference. We briefly recall these notions in our context of Coxeter groups.  

Let $(W,S)$ be a Coxeter system. Let $w\in W$. The \textit{cone of $(W,S)$ with root $w$} is the set
$$
C_W(w):=\{v\in W\mid d(e,v)=d(e,w)+d(w,v)\}.
$$
That is, $C_W(w)$ is the set of all elements $v\in W$ such that there exists a geodesic from $e$ to $v$ passing through $w$. The \textit{cone type} of the cone $C_W(w)$ is
$$
T_W(w):=\{v\in W \mid wv\in C_W(w)\}=w^{-1}C_W(w).
$$

Let $\mathcal{T}(W,S)$ be the set of cone types of $(W,S)$. By Brink and Howlett \cite[Theorem~2.8]{howlett} there are only finitely many cone types in a Coxeter system $(W,S)$, and so $|\mathcal{T}(W,S)|<\infty$. 

\begin{defn}\label{defn:cannon} The \textit{Cannon automaton} of the Coxeter system $(W,S)$ is the directed graph $\mathcal{A}(W,S)$ with vertex set $\mathcal{T}(W,S)$ and with labelled edges defined as follows: there is a directed edge with label $s\in S$ from cone type $\bT$ to cone type $\bT'$ if and only if there exists $w\in W$ such that $\bT=T_W(w)$ and $\bT'=T_W(ws)$ and $d(e,ws)=d(e,w)+1$.
\end{defn}

Let $G=(V_G,E_G,s_\ast)$ be a finite, directed graph with vertex set $V_G$,
edge set $E_G$ and with some distinguished vertex
$s_\ast\in V_G$ together with a labelling $\alpha: E_G\to S$ of the edges.
A \textit{path} in $G$ is a sequence of vertices $[\x_0,\x_1,\dots,\x_m]$,
$m\in\N$, such that there is a directed edge from $\x_{i-1}$ to $\x_i$ for
all $i\in\{1,\dots,m\}$. A \textit{ray} is an infinite path of the form
$[\x_0,\x_1,\dots ]$ such that there is a directed edge from
$\x_{i-1}$ to $\x_i$ for all $i\in\N$. Sometimes we will identify paths or rays
by their set of vertices $\{\x_0,\x_1,\dots\}$.
Denote by $\Paths$ the set of  all finite paths in $G$ starting at $s_*$ and
by $\Paths_\infty$ the set of all rays in $G$ starting at $s_*$.
Each path  of the form
$\gamma=[s_*,\x_1,\ldots, \x_{m}] \in \Paths$,
$\gamma=[s_*,\x_{1},\ldots ]\in\Paths_\infty$ respectively,
gives rise to a path in $W$ starting from $e$. Denote by $\e{1}$ the edge between $s_{*}$ and $\x_{1}$ and by  $\e{i}$ the edge between $\x_{i-1}$ and $\x_{i}$ for $i\geq 2$; the path in $W$ corresponding to $\gamma$ is then defined by 
$$ 
\alpha(\gamma) :=\Bigl[ e, \alpha(\e{1}),\alpha(\e{1})\alpha(\e{2}),\ldots,
\prod_{i=1}^{m} \alpha(\e{i})\Bigr],\quad 
\alpha(\gamma) :=\bigl[ e,
\alpha(\e{1}),\alpha(\e{1})\alpha(\e{2}),\ldots \bigr]\textrm{ respectively.}
$$
For $\gamma\in \Paths$, we define $\alpha^\ast(\g):=\prod_{i=1}^{m}\alpha(\e{i})$.

\begin{defn}
An \textit{automatic structure} for $(W,S)$ is a finite, directed graph
$\cA=(V, E, s_*)$ along with a labelling $\alpha: E\to S$ such that:
\begin{enumerate}
\item No edge in $E$ ends at $s_*$.
\item For $\mathbf{v}\in V\setminus \{s_\ast\}$, there is a path in $\cA$ from $s_*$ to $\mathbf{v}$.
\item For every path $\gamma\in \Paths$ in $\cA$, the path $\alpha(\gamma)$ is a geodesic in $W$.
\item The mapping $\alpha^{*}:\Paths \to W$  is surjective.
\end{enumerate}
The graph $\cA$ is called a \textit{finite state automaton}. We speak of a \emph{strongly automatic structure} if $\alpha^{*}$ defines a bijection between $\Paths$ and $W$.
\end{defn}
It is easy to check that the Cannon automaton $\A(W,S)$, equipped with the natural labelling induced from the labels on the edges, is an automatic structure for $(W,S)$ with $s_\ast=T(e)$.
We may obtain a strongly automatic structure $\mathfrak{A}=\mathfrak{A}(W,S)$
from $\A(W,S)$ by choosing a
lexicographic order of the cone types; see \cite[Section 2.5]{cannon} and Appendix~\ref{app:A} for details. 

The strongly automatic structure $\mfA$ allows us to define \emph{cones}
$C_{\mfA}$ and \emph{cone types} $T_{\mfA}$ as follows. For given $w\in W$ and
a path $\gamma_w\in\Paths$ in $\mfA$ with $\alpha^\ast(\gamma_w)=w$, we write $C_{\mfA}(w)$ for the set of all elements of $W$
which are represented by paths in $\mfA$ with $\gamma_w$ as prefix, that is,
$$
C_\mfA(w):=\{u\in W \mid \exists \gamma\in\Paths :
\alpha^\ast(\gamma)=u,\gamma \text{ goes through }w\}.
$$ 
We call $C_\mfA(w)$ the \textit{cone with respect to~$\mfA$ rooted at $w$}, which is well defined since the path $\gamma_{w}$ is
uniquely determined. We say that $w$ is the \emph{root} of the cone
$C_\mfA(w)$. Moreover, we can define \textit{cone types with respect to~$\mfA$} as $T_\mfA(w):=w^{-1}C_\mfA(w)$. Note that the cones $C(w)$ and
$C_\mfA(w)$ are not necessarily equal. The important property is now that
$C_\mfA(u)\cap C_\mfA(v)=\emptyset$ if $u\neq v$ and $d(e,u)=d(e,v)$. Thus two cones
with respect to~$\mfA$ are
either disjoint or nested in each other ($C_\mfA(u)\subseteq C_\mfA(v)$ or
$C_\mfA(v)\subseteq C_\mfA(u)$). Obviously, we still have only finitely many cone types
with respect to~$\mfA$. From now on, when we speak of cones and cone
types we always think of the cones and types with respect to~the strongly automatic
structure $\mfA$ unless stated otherwise. For sake of brevity we write
$C(w):=C_\mfA(w)$ and $T(w):=T_\mfA(w)$.

A cone type $\bT'$ is \textit{accessible} from the cone type $\bT$ if there is a path from $\bT$ to $\bT'$ in the (directed) graph $\mfA$. In this case we 
write $\bT\to \bT'$. A cone type $\bT$ is called \textit{recurrent} if $\bT\to\bT$, and otherwise it is called \textit{transient}. The set of 
recurrent vertices induces a (directed) subgraph $\mfA_R$ of $\mfA$.
We call the automaton $\mfA$ \textit{strongly connected} if each recurrent cone type is accessible from any other recurrent cone type in the subgraph~$\mfA_R$, that is, the subgraph $\mfA_R$ consists of one connected component. We will sometimes identify $\mfA_R$ with its vertex set.

Recurrence and accessibility of cone types with respect to~the Cannon automaton $\A$ is
defined analogously to $\mfA$. By \cite[Theorem 3.2]{gilch-mueller-parkinson:14}, the subgraph $\A_{\mathcal{R}}$ of $\A$ is strongly connected. Moreover, we have:
\begin{thm}\label{thm:strongly-connected}
Let $(W,S)$ be an irreducible Fuchsian Coxeter system. The strongly automatic structure $\mfA$ for $(W,S)$ is strongly connected.
\end{thm}
\begin{proof}
We provide the proof in Appendix~\ref{app:A}. 
\end{proof}
A first consequence of this  theorem is that there is some $K\in\mathbb{N}$ such that
$W\setminus \cB_K(e)$ contains only elements of recurrent cone types.
\par
We give further useful definitions related to cones and cone types.
The (inner) \textit{boundary} of a cone $C(w)$ of $(W,S)$ is
$$
\partial C(w):=\{u\in C(w)\mid \exists v\in W\setminus C(w): d(u,v)=1\}.
$$
More generally, if $L\geq 1$, the \textit{$L$-boundary} of a cone $C(w)$ of $(W,S)$ is defined to be
$$
\partial_L C(w):=\{u\in C(w)\mid \exists v\in W\setminus C(w): d(u,v)\leq L\}.
$$
We call $\mathrm{Int}_{L} C(w):=C(w)\setminus \partial_{L} C(w)$ the
\textit{$L$-interior} of $C(w)$. See also Figure \ref{fig:def:con}. The \textit{$L$-boundary of a cone type}
$\bT(w)$ of $(W,S)$ with respect to~$\mfA$ is defined  by
$$
\partial_L\bT(w):=\{u\in\bT(w)\mid \exists v\in W\setminus\bT(w): d(u,v)\leq L\},
$$
and the \textit{$L$-interior of the cone type} $\bT(w)$ is $\mathrm{Int}_L
\bT(w):=\bT(w)\setminus \partial_L\bT(w)$.
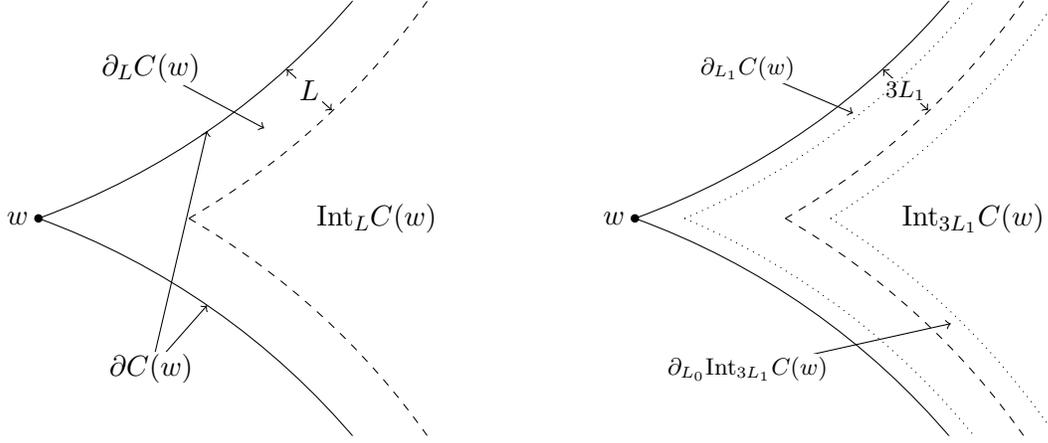
\begin{figure}
\centering
\hspace{0.5cm}
\subfigure{
\begin{tikzpicture}[scale=2]
  \begin{scope}
  \draw (0,0) arc (-70:-40:5);
  \end{scope}
  \begin{scope}[yscale=-1]
  \draw (0,0) arc (-70:-40:5);
  \end{scope}
  \fill(0,0) circle[radius=0.8pt, fill=black];
  \draw(0,0)node[left]{$w$};
  
  \tikzstyle{ann} = [fill=white,inner sep=1pt]
 \draw[arrows=<->](1.65,0.99)--(1.95,0.72);
 \node[ann] at (1.8,0.85) {$L$};
   \begin{scope}[xshift=1cm, dashed]
  \draw (0,0) arc (-60:-35:5);
  \end{scope}
  \begin{scope}[yscale=-1, xshift=1cm, dashed]
  \draw (0,0) arc (-60:-35:5);
   \draw(1.25,0)node[]{$\mathrm{Int}_{L} C(w)$};
  \end{scope}
 
 \draw[arrows=->](0.75,1)--(1.5,0.6);
 \node[ann] at (0.75,1) {$\partial_{L}C(w)$};
 
 \draw[arrows=->](0.75,-1)--(1.12,-0.58);
 \draw[arrows=->](0.75,-1)--(1.12,0.58);
 \node[ann] at (0.75,-1) {$\partial C(w)$};
 
\filldraw[fill=white, draw=white] (0,1.45) rectangle (3,1.5);
\filldraw[fill=white, draw=white] (0,-1.45) rectangle (3,-1.5);
 \end{tikzpicture} }
 \hfill
 \subfigure{
 \begin{tikzpicture}[scale=2]
  \begin{scope}
  \draw (0,0) arc (-70:-40:5);
  \end{scope}
  \begin{scope}[yscale=-1]
  \draw (0,0) arc (-70:-40:5);
  \end{scope}
  \draw(0,0)node[left]{$w$};
   \fill(0,0) circle[radius=0.8pt, fill=black];
  \tikzstyle{ann} = [fill=white,font=\footnotesize,inner sep=1pt]
 \draw[arrows=<->](1.65,0.99)--(1.95,0.72);
 \node[ann] at (1.8,0.85) {$3L_{1}$};
   \begin{scope}[xshift=1cm, dashed]
  \draw (0,0) arc (-60:-35:5);
  \end{scope}
  \begin{scope}[yscale=-1, xshift=1cm, dashed]
  \draw (0,0) arc (-60:-35:5);
   \draw(1.25,0)node[]{$\mathrm{Int}_{3L_1} C(w)$};
  \end{scope}
 
  \begin{scope}[xshift=1.3cm, dotted]
  \draw (0,0) arc (-58:-34:5);
  \end{scope}
  \begin{scope}[yscale=-1, xshift=1.3cm, dotted]
  \draw (0,0) arc (-58:-34:5);
  \end{scope}
  \draw[arrows=->](0.75,-1)--(2.1,-0.7);
 \node[ann] at (0.75,-1) {$\partial_{L_{0}}\mathrm{Int}_{3L_1} C(w)$};
 
 \begin{scope}[xshift=0.33cm, dotted]
  \draw (0,0) arc (-67:-38.5:5);
  \end{scope}
  \begin{scope}[yscale=-1, xshift=0.33cm, dotted]
  \draw (0,0) arc (-67:-38.5:5);
  \end{scope}
  \draw[arrows=->](0.75,1)--(1.45,0.7);
 \node[ann] at (0.75,1) {$\partial_{L_{1}}C(w)$};
 
\filldraw[fill=white, draw=white] (0,1.45) rectangle (3,1.5);
\filldraw[fill=white, draw=white] (0,-1.45) rectangle (3,-1.5);
 \end{tikzpicture} 
 }
 \hspace{0.5cm}
 \caption{Illustrations of the cone $C(w)$ and its different boundaries.}
 \label{fig:def:con}
\end{figure}

Using planarity of the Cayley graph of $(W,S)$ we have the following important description for the boundary of a cone $C(w)$,
see Ha\"issinsky, Mathieu and M\"uller \cite[Lemma 2.4]{HMM:13} and compare with \cite[Lemma 5.1]{gilch-mueller-parkinson:14}.
\begin{lemma}\label{lem:boundary-rays}
Let $u\in W\setminus \{e\}$. Then the boundary $\partial C(u)$ is
contained in the union of two geodesic rays in $W$ starting at~$u$ which can
also be described
by rays in $\mfA$.
\end{lemma}

We  recall the well-known crucial fact that geodesics in a hyperbolic group
either stay within bounded distance of each other or diverge
exponentially.  More precisely, there exists some exponential divergence
function $e: \NN_0 \rightarrow \RR$ such that the following holds.  Let $u,v_{1},v_{2}\in W$ and 
$\gamma_{1}$ a geodesic from $u$ to  $v_{1}$ and $\gamma_{2}$ a geodesic from $u$ to $v_2$. We denote by $\gamma_i(n)$  the point on $\gamma_i$ at distance $n\in\NN_0$ to $u$.  Then, for  all $r,R\in \NN_0$ with $R+r\leq \min\{d(u,v_{1}), d(u,v_{2})\}$
and $d(\gamma_{1}(R), \gamma_{2}(R))\geq e(0)$, every path $\eta$ starting in $\gamma_{1}(R+r)$, visiting only vertices in $W\setminus
\cB_{R+r}(u)$ and ending in $\gamma_{2}(R+r)$ has length of
at least $e(r)$, see also Figure \ref{fig:exp div}. In particular, two geodesics that have been at least $e(0)$ apart can never
intersect again. 
\begin{figure}
\begin{center}
\begin{tikzpicture}[scale=2]
  \tikzstyle{ann} = [fill=white,font=\footnotesize,inner sep=1pt]
  \begin{scope}
  \draw (0,0) arc (-70:-40:5);
  \end{scope}
  \begin{scope}[yscale=-1]
  \draw (0,0) arc (-70:-40:5);
  \end{scope}
  \fill(0,0) circle[radius=0.8pt, fill=black];
  \fill(2.1,1.47) circle[radius=0.8pt, fill=black];
\fill(2.1,-1.47) circle[radius=0.8pt, fill=black];
  \draw(0,0)node[left]{$u$};
  \draw(2.1,1.47)node[label=right:$v_{1}$]{};
  \draw(2.1,-1.47)node[label=right:$v_{2}$]{};
\draw (0,0) ++ (-40:1.75cm) arc (-40:40:1.75cm);
\draw[<->]  (0.7,0.28) --(0.7,-0.28);
 \node[ann] at (0.7,0) {$> e(0)$};
 \fill(0.7,0.315) circle[radius=0.8pt, fill=black];
 \fill(0.7,-0.315) circle[radius=0.8pt, fill=black];
\draw(0.9,-1.2)node[right]{$B_{R+r}(u)$};
\draw(0.65,0.34)node[above]{$\gamma_{1}(R)$};
\draw(0.65,-0.34)node[below]{$\gamma_{2}(R)$};
\fill(1.52,0.88) circle[radius=0.8pt, fill=black];
\fill(1.52,-0.88) circle[radius=0.8pt, fill=black];
\draw(1.52,0.88)node[right]{$\gamma_{1}(R+r)$};
\draw(1.52,-0.88)node[right]{$\gamma_{2}(R+r)$};
\draw plot [smooth, tension=1] coordinates { (1.52,0.88) (1.8,0.6) (2.1,0.2) (1.8,-0.2)  (2,-0.55)(1.52,-0.88)};
\draw(2.1,0.2) node[right]{$\eta$};
\end{tikzpicture}
\end{center}
\caption{Geodesics diverge exponentially: length of path $\eta$ is bigger than $e(r)$.}
\label{fig:exp div}
\end{figure}
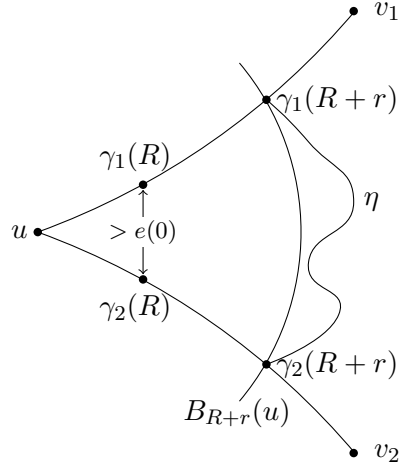
Furthermore, for each cone of recurrent type the two geodesics which describe
its boundary diverge exponentially. This follows from exponential growth of $(W,S)$ which yields that $\mfA_R$ is not a circle.

\subsection{Isotropic random walks on regular buildings}

We will henceforth write $(\Delta,\delta)$ for a thick regular building of type
$(W,S)$, where $(W,S)$ is an irreducible Fuchsian Coxeter system. A (time-homogeneous) random walk $(X_n)_{n\geq 0}$ on the set $\Delta$ of chambers of the
building $(\Delta,\delta)$ is \textit{isotropic} if the single-step transition probabilities \mbox{$p(x,y):=\mathbb{P}[X_{n+1}=y\mid X_n=x]$} of the walk satisfy
$$
p(x,y)=p(x',y')\quad\textrm{whenever $\delta(x,y)=\delta(x',y')$}. 
$$
In other words, the probability of jumping from chamber $x$ to chamber $y$ in one step depends only on the Weyl distance $\delta(x,y)$. Thus an isotropic random walk is determined by the probabilities
$$
p_w:=\mathbb{P}[X_1\in \Delta_w(o)\mid X_0=o],\quad\textrm{so that}\quad p(x,y)=p_w/q_{w}\quad\textrm{if $\delta(x,y)=w$}.
$$
We will typically stipulate $X_0=o$ (the fixed base chamber), and we will always assume that $(X_n)_{n\geq 0}$ is irreducible. Since $\Delta$ is thick, by \cite[Lemma 4.2]{gilch-mueller-parkinson:14} the irreducible isotropic random walk $(X_n)_{n\geq 0}$ on $\Delta$ is necessarily aperiodic. 

For each $n\geq 0$ the $n$-step transition probabilities are denoted by
$$
p^{(n)}(x,y):=\mathbb{P}[X_n=y\mid X_0=x] \textrm{ and } p_w^{(n)}:=\mathbb{P}[X_n\in \Delta_w(o)\mid X_0=o].
$$
Then 
$p^{(n)}(x,y)=p^{(n)}_w/q_w$ whenever $\delta(x,y)=w$. For $x,y\in\Delta$, the \textit{Green function} is then given by
$$
G(x,y):=\sum_{n\geq 0} p^{(n)}(x,y).
$$
By irreducibility, the spectral radius of~$(X_n)_{n\geq 0}$ is given by
$$
\varrho(P):=\limsup_{n\to\infty}p^{(n)}(x,y),
$$
and this value does not depend on the pair $x,y\in\Delta$, and
by \cite[Corollary 4.9]{gilch-mueller-parkinson:14} we have
$\varrho(P)<1$. This implies transience of $(X_n)_{n\geq 0}$.

We will assume that $(X_n)_{n\geq 0}$ has \textit{bounded
range}, that is, there is a minimal number $L_0\geq 0$ such that 
\begin{equation}
\label{eq:L0}\text{$p_w\neq 0$ implies that $\ell(w)\leq L_0$}.
\end{equation}
We set $\varepsilon_0:=\min\{ p_w/q_w \mid w\in W, p_w>0\}$. Thus $\varepsilon_0>0$.

Let $\pi_n$ be the distribution of $X_n$.
The \textit{asymptotic entropy} of the random walk $(X_n)_{n\geq 0}$ is given by
$$
h := \lim_{n\to\infty} -\frac1n \mathbb{E}[\log \pi_n(X_n)],
$$
if the limit exists. The main goal of this article is to prove that this limit
exists, and to give formulae for it.

It is sufficient to prove our results under the assumption that $p_s>0$ for all $s\in S$. To see this, note that by \cite[Lemma 4.3.2]{gilch-mueller-parkinson:14} there is an $M\geq 1$ such
that the $M$-step walk $(X_{nM})_{n\geq 0}$ satisfies $p_s^{(M)}>0$ for all
$s\in S$, and by the bounded range assumption proving existence of the limit
$\lim_{n\to\infty}-\frac{1}{nM}\mathbb{E}\bigl[\log \pi_{nM}(X_{nM})\bigr]$ 
 implies that $-\frac1n\mathbb{E}\bigl[\log
 \pi_{n}(X_{n})\bigr]$ converges to the same limit due to
$$
\pi_{nM}(X_{nM}) \cdot \varepsilon_0^{k} \leq \pi_{nM+k}(X_{nM+k}) \leq
\frac{1}{\varepsilon_0^{M-k}}\pi_{(n+1)M}(X_{(n+1)M}) \quad \textrm{a.s. for } k\in\{0,1,\dots,M\}.
$$
Thus, without loss of generality we will assume throughout the paper due to (\ref{eq:L0}) that 
\begin{align}
\label{eq:support}\text{$p(x,y)\geq \varepsilon_0>0$ whenever $d(x,y)=1$}.
\end{align}
 
\par
Another fundamental statistic of the random walk $(X_n)_{n\geq 0}$ is the \textit{rate of escape} or \textit{drift} which is given by 
\begin{equation}\label{equ:speed}
\mathtt{v}:=\lim_{n\to\infty} \frac{d(o,X_n)}{n}.
\end{equation}
This limit exists almost surely and is constant, see \cite[Theorem
1.1]{gilch-mueller-parkinson:14}. The \textit{rate of escape with respect to the Green distance} is given by
$$
\lim_{n\to\infty} -\frac1n \log G(o,X_n),
$$
if the limit exists. In Corollary~\ref{cor:buildinggreen} we show that this limit exists and equals the asymptotic entropy $h$.

\subsection{Retracted walk}

The main tool for our proof of existence of the asymptotic entropy of the
random walk $(X_n)_{n \geq 0}$ on $(\Delta,\delta)$  is to look at the image
$\overline{X}_n:=\rho(X_n)$ of the random walk under the retraction
\mbox{$\rho:\Delta\to W$.} In \cite[Proposition 4.5]{gilch-mueller-parkinson:14} we
have shown that the stochastic process $(\overline{X}_n)_{n\geq 0}$ on $W$ is
in fact a random walk on~$W$, which we call the \textit{retracted
  walk}. We denote analogously the single-step transition probabilities by
$\overline{p}(u,v)$ and the $n$-step transition probabilities by
$\overline{p}^{(n)}(u,v)$, and we will use the notation $\mathbb{P}_u[\,\cdot\,] := \mathbb{P}[\,\cdot\,|\,
\overline{X}_0=u]$. However we note that the retracted walk is not $W$-invariant. That is,
$\overline{p}(wu,wv)\neq \overline{p}(u,v)$ in general. But we have the following weaker invariance property which roughly says that the transition probabilities of the retracted walk in two cones of the same type are the same:
\begin{prop}\label{prop:invariance}
Let $\bT$ be a cone type of $(W,S)$ with respect to~the strongly automatic structure
$\mfA$ and $w_{1},w_{2}\in W$ with $T(w_1)=T(w_2)=\bT$.
Then
$$
\overline{p}(w_1u,w_1v)=\overline{p}(w_2u,w_2v)\quad\textrm{for all $u\in \bT$ and all $v\in \mathrm{Int}_{L_0} \bT=\bT\setminus\partial_{L_0}\bT$}.
$$
\end{prop}
\begin{proof} 
The proof follows directly from \cite[Proposition
4.7]{gilch-mueller-parkinson:14}, where these equations were shown for the
cones with respect to~the Cannon automaton $\A$; the equations also hold for the cones with respect to~the strongly automatic structure $\mfA$ since $C_\mfA(w)\subseteq C_W(w)$ and also $\mathrm{Int}_{L_0} C_\mfA(w) \subseteq \mathrm{Int}_{L_0} C_W(w)$. 
\end{proof}

The following important property is proven analogously to \cite[Lemma 5.8]{gilch-mueller-parkinson:14} by noting that there are only finitely many cone types.
\begin{lemma}\label{lem:stayincone}
There exists a
constant $\overline{p}_{esc}>0$ such that for all 
$w\in W$  
$$
\PP_{w}[\overline{X}_n\in C(w)\text{ for all $n\geq 0$}]\ge\overline{p}_{esc}.
$$
\end{lemma}
Let $\overline{P}:=\bigl(\overline{p}(u,v)\bigr)_{u,v\in W}$ be the transition
operator of the retracted walk, and let
$\varrho(\overline{P}):=\limsup_{n\to\infty} \overline{p}^{(n)}(u,v)^{1/n}$ be the
spectral radius of $\overline{P}$, which is independent of the specific choice
of $u$ and $v$. Then $\varrho(\overline{P})=\varrho(P)<1$ due to
\cite[Proposition 4.6]{gilch-mueller-parkinson:14}. Since
$d(o,X_n)=\ell(\delta(o,X_n))=\ell(\rho(X_n))=d(e,\overline{X}_n)$, we also have
$\mathtt{v}=\lim_{n\to\infty} \frac1n \ell(\overline{X}_n)$, with $\mathtt{v}$ as in (\ref{equ:speed}). 
\par
Analogously to the random walk $(X_n)_{n\geq 0}$ on $(\Delta,\delta)$ let 
$\overline{\pi}_n$ be the distribution of $\overline{X}_n$. The \textit{asymptotic
  entropy} of the retracted walk $(\overline{X}_n)_{n\geq 0}$ is then given by
$$
\overline{h} := \lim_{n\to\infty} -\frac1n \mathbb{E}[\log \overline{\pi}_n(\overline{X}_n)],
$$
if the limit exists.


\section{Asymptotic entropy of the retracted walk}
\label{sec:entropy-retracted-walk}

In this section we derive a formula for the asymptotic entropy of the retracted
walk (see Theorem~\ref{thm:L-limit} together with (\ref{equ:h_G})). We begin by introducing some generating functions and notation in
Subsection \ref{sec:genfun}. In Subsection \ref{sec:cone-coverings} we will construct nested
sequences of coverings of $W$ by cones such that we can track the retracted
random walk's way to infinity. In Subsections \ref{sec:entropy-HMC}, \ref{sec:entropy-retracted} and \ref{subsec:entropy-retracted-walk} we will use
this construction of coverings by cones in order to deduce existence and
formulae for the asymptotic entropy of the retracted walk $(\overline{X}_n)_{n\geq 0}$.

\subsection{Generating functions}

\label{sec:genfun}

In this section we define some useful generating functions and some related notation.
For $u,v\in W$, $z\in\mathbb{C}$, the \textit{Green function} of the retracted
walk is defined as 
$$
\overline{G}(u,v|z) := \sum_{n\geq 0} \overline{p}^{(n)}(u,v)\cdot z^n
$$
and the \textit{last visit generating function} is given by
$$
\overline{L}(u,v|z) := \sum_{n\geq 0}\Prob_u\bigl[\overline{X}_n=v,\forall m\in\{1,\dots,n\}:
\overline{X}_m\neq u \Bigr]\cdot z^n.
$$
We will write $\overline{G}(u,v):=\overline{G}(u,v|1)$.
By conditioning on the last visit to $u$, a fundamental relation
between these functions is given by
\begin{equation}\label{equ:G-L-splitting}
\overline{G}(u,v|z) = \overline{G}(u,u|z) \cdot \overline{L}(u,v|z).
\end{equation}
Since $(\overline{X}_n)_{n\geq 0}$ is irreducible and has spectral radius
strictly smaller than $1$ the Green functions have a common radius of convergence $R>1$. 
From \cite[Proposition 5.11]{gilch-mueller-parkinson:14} it follows that there is
some $\lambda\in(0,1)$ and some $C>0$ such that for all $u,v\in W$
\begin{equation}\label{equ:Green-estimate}
\overline{G}(u,v)\leq C \cdot \lambda^{d(u,v)}.
\end{equation}
For $v,w\in W$, the \textit{Green distance} is defined as $d_{G}(v,w):=-\log \frac{\overline{G}(v,w)}{\overline{G}(v,v)}$ and we write 
\begin{equation*}
\ell_{G}(w):=d_{G}(e,w).
\end{equation*}
We will show that the limit
$$
\lim_{n\to\infty} \frac{\ell_{G}(\overline{X}_n)}{n}
$$
exists almost surely, and equals the asymptotic entropy $\overline{h}$ (see Theorem~\ref{thm:L-limit}). This limit is called the \textit{rate of escape with respect to the Green distance} of
the retracted walk.

\subsection{Cone covering and last entry times of cones}
\label{sec:cone-coverings}
In order to trace the retracted walk's path to infinity we define inductively a
sequence of nested cones. For this purpose, we use the strongly
automatic structure $\mfA$ for $(W,S)$, which implies a partial order
$\preceq$ on the group elements, \emph{i.e.},
$u\preceq v$ if and only if $C(u)\subseteq C(v)$ for
$u,v\in W$. 
At this point we will make use of the fact that $\mfA_{R}$ is strongly
connected; see Theorem \ref{thm:strongly-connected}.
Furthermore, we use the fact that the subgraph $\mfA_{R}$ is not a circle, \emph{i.e.},
it contains at least one vertex of outdegree of at least $2$ (otherwise $W$
would not have exponential growth). 

Recall that the Cayley graph of $(W,S)$ is a $\delta$-hyperbolic space and hence
  triangles are $\delta$-thin. Note that in this context the ``$\delta$''
  differs from the ``$\delta$'' denoting the Weyl distance function. Define$^1$\footnote{$^1$The divergence function $e(\cdot)$ can be chosen such that $e(0)=\delta$, hence one could also set $L_{1}:=\max\{L_0,\delta\}+1$.} 
  $$L_1:=\max\{e(0),L_0,\delta\}+1.$$ 
\begin{lem}\label{lem:cone non intersection}
Let $w\in W$ and $v\in \mathrm{Int}_{3 L_1}C(w)$. Then $C(v)\cap\partial_{L_1} C(w)=\emptyset.$
\end{lem}
\begin{proof}
Let $u\in \partial C(w)$ and $u'\in C(v)$, and consider the following geodesic
triangle: let $\gamma$ be a geodesic from $w$ to $u$, $\gamma'$ be a geodesic
from $w$ to $u'$ passing through $v$, and $\tilde \gamma$ be a geodesic from
$u$ to $u'$. Let $p\in \gamma'$ such that $p\notin C(v)$ and $d(p,v)=2 L_1$. As $p$ is not in the
$\delta$-neighborhood of $\gamma$ (recall that $L_1>\delta$ and
$v\in\mathrm{Int}_{3L_1} C(w)$) there exists
some point $\tilde p\in \tilde \gamma$ such that $d(p,\tilde p)\leq \delta <L_1$. Since
$u'\in C(v)$ we get $d(\tilde p, u')>L_1\geq \delta$ and hence $d(u,u')>L_1$.
\end{proof}

\begin{prop}\label{prop:covering construction}
For every $\bT\in\mfA_{R}$ there exists a set $\mathrm{Cov}(\bT):=\{u_{1},u_{2},\ldots \}\subset
W$ such that, for all $w\in W$ with $T(w)=\bT$, the following properties hold:
\begin{enumerate}
\item $wu_i\in\mathrm{Int}_{3L_1}C(w)$ for all $i\in \N$,
\item $C(wu_{i})\subseteq \mathrm{Int}_{L_1} C(w)$ for all $i\in \N$,
\item $C(wu_i)\cap C(wu_j)=\emptyset$ for $i\neq j$,
\item $ \{ T(wu_{1}), T(wu_{2}),  \ldots\}=\mfA_{R}$,
\item $C(w) \setminus \bigcup_{i\geq 1} C(wu_{i})\subset \partial_{3L_{1}} C(w)\cup \cB_{L}(w)$ for some constant $L=L(\bT)$,
\item $\mathrm{Cov}(\bT)\cap \cB_{n}(e)$ grows linearly in $n$.
\end{enumerate}
\end{prop}
\begin{proof}
Let be $\bT\in\mfA_{R}$ and $w\in W$ with $T(w)=\bT$. 
We enumerate the elements in $\mfA_{R}$ by  $\bT_{1},\bT_{2},\ldots,\bT_{k}$ where $k=| \mfA_{R} |$. There is some  $N\in\N$ large enough
such that there are $w_1,w'\in\mathrm{Int}_{3L_1} C(w)$ with $w_1\neq
w'$ and $d(w,w_1)=d(w,w')=N$ and $\bT(w_{1})=\bT_{1}$. 
We then have $C(w_1)\cap C(w')=\emptyset$, and we have found a
cone of type $\bT_{1}$, set $u_1:=w^{-1}w_1$ and add $u_1$ to $\mathrm{Cov}(\bT)$. In the same way we  search
for a cone of type $\bT_{2}$ in $C(w')$ and find $u_{2}$ that we add to our
covering $\mathrm{Cov}(\bT)$. This procedure is again repeated $k-2$ times such
that the covering $\mathrm{Cov}(\bT)$ contains all possible cone types, see
property \emph{4}. We now ``fill'' up the covering in order to ensure property
\emph{5}. For  $n\in\mathbb{N}$, set $\mathcal{S}_n(w):=\{w_0\in W \mid d(w,w_0)=n\}$.
 We start with $n=L:=\max\{|u_1|,\dots,|u_k|\}$: for each $\bar w \in\mathcal{S}_{n}(w)\cap \mathrm{Int}_{3L_1}C(w)$ such that
$\bar w\notin C(wu)$ for all $u\in \mathrm{Cov}(\bT)$ we add the element $w^{-1}\bar w$ to the covering $\mathrm{Cov}(\bT)$. Inductively
this is repeated  for all $n>L$. The resulting covering $\mathrm{Cov}(\bT)$
verifies properties \emph{1}, \emph{3}, and \emph{5}. Property \emph{2}~holds due to Lemma \ref{lem:cone
  non intersection} and since $wu_i\in\mathrm{Int}_{3L_1} C(w)$. Property \emph{6}~is a consequence of the planarity of
the Cayley graph. For $n$ sufficiently large 
 any $u'\in\mathrm{Cov}(\bT)\cap
\cB_{n}(e)$ has to be of the form $u'=\tilde u s$ with $\tilde u \in \partial_{3L_1}
C(w)\cap B_{n-1}(w)$ and $s\in S$. Now, since $\partial_{3L_1}
C(w)\cap B_{n-1}(w)$ grows linearly property \emph{6} follows.
\end{proof}

For each $\bT\in\mfA_R$ we now fix a covering $\mathrm{Cov}(\bT)$ satisfying
the five properties in Proposition~\ref{prop:covering construction}. We write
$\mathrm{Cov}(w):=w\mathrm{Cov}(\bT)=\{wu \mid u\in\mathrm{Cov}(\bT)\}$ for $w\in W$ with $T(w)=\bT$. 

We now define a covering of $W$  by induction.
Let be $K\in\N$ such that $ W\setminus \mathcal{B}_K(e)$ contains only recurrent cone
types. Define $M_0:=\mathcal{B}_{K+1}(e)\setminus \mathcal{B}_K(e)$ and set $\mathrm{Cov}_{0}:=\bigcup_{w\in M_0} \mathrm{Cov}(w)$. Note that we have that the cones $C(w),~w\in \mathrm{Cov}_{0},$ are pairwise disjoint. Furthermore, \mbox{$\mfA_R=\{T(w)\mid w\in \mathrm{Cov}_0\}$.} Given the set $\mathrm{Cov}_{n}$ for
some $n\in\N$, we define inductively
\[
\mathrm{Cov}_{n+1}:=\bigcup_{w\in \mathrm{Cov}_{n}} \mathrm{Cov}(w).
\]
We set $\mathrm{Cov}:=\bigcup_{n\geq 0} \mathrm{Cov}_{n}$. 
The next lemma states that the elements of $\mathrm{Cov}$ are (in a certain sense) dense in $W$.


\begin{lemma}\label{lem:cover-dense}
\begin{enumerate}
\item There exists some $K\in\mathbb{N}$ such that for every $w_0\in W$ and all $w\in C(w_0)$ there exists some
$v\in\mathrm{Cov}\cap C(w_0)$ with $d(w,v)\leq K$.
\item There exists an integer $K_1\in\mathbb{N}$ such that the following holds:
  if $w_0\in\mathrm{Cov}_k$, $k\geq 1$, and $w\in\partial_{L_0}
  \mathrm{Int}_{3L_1} C(w_0)$ 
then there exists $v\in
  \mathrm{Cov}_{k+2}\cap C(w_0)$ such that $v$ can be reached from
  $w$ by a path inside $\mathrm{Int}_{L_1} C(w_0)$ of length at most $K_1$.
\end{enumerate}
\end{lemma}
\begin{proof}
Let $L$ be the maximal constant from the proof of Proposition \ref{prop:covering construction} (when varying through recurrent cone types). We consider three cases in order to prove the first assertion. 

\noindent \emph{Case 1: }
If $w\in C(w_0)\cap \cB_L(w_0)$ then $d(w,\mathrm{Cov})\leq
d(w,w_0)\leq L$.

\noindent \emph{Case 2: }
If $w\in C(w_0)\setminus \cB_L(w_0)$ and $w\notin \bigcup_{v\in \mathrm{Cov}(w_0)} C(v)$  then  $w\in\partial_{3L_1} C(w_0)$. Thus, there is some $\hat
w\in\partial C(w_0)\setminus \cB_L(w_0)$ with $d(w,\hat w)\leq 6L_1$. Consider
the cone $C(\hat w)$ with boundary geodesics $\gamma_{1}$ and $\gamma_{2}$. One of these geodesics, say  $\gamma_1=[\hat w,\hat w g_1,\hat
w g_2,\dots]$, intersects $\mathrm{Int}_{3L_1}C(w_0)$. Now, 
there exists some $\hat wg_{i_{T(\hat w)}}\in\gamma_1$ of minimal index ${i_{T(\hat w)}}$
such that $\hat wg_{i_{T(\hat w)}}\in\mathrm{Int}_{3L_1}C(w_0)$, and this index ${i_{T(\hat w)}}$ depends
only on the cone type $T(\hat w)$ and not on $\hat w$ itself. Since we have a
strongly automatic structure, we must have $\hat wg_{i_{T(\hat w)}}\in \mathrm{Cov}$, which
yields $d(w,\hat wg_{i_{T(\hat x)}})\leq 6L_1+{i_{T(\hat w)}}$. Since we have only finitely many one types
we get the claim also in this case, and we may set $K:=\max \{L,6L_1+i_{\bT}
\mid \bT\in\mfA_R\}$.

\noindent \emph{Case 3: }
If $w\in \bigcup_{v\in \mathrm{Cov}(w_0)} C(v)$ and $d(w_0,w)>L$ then we 
choose $w'\in \mathrm{Cov}(w_0)$ with $w\in C(w')$ and we exchange $w_0$ by $w'$,
and we iterate the proof from the beginning until $d(w_0,w)\leq L$ \mbox{(case 1)} or $w\notin \bigcup_{v\in \mathrm{Cov}(w_0)} C(v)$ (case 2). This proves part \emph{1}.

For the proof of part \emph{2} we first consider the case that $w\in\partial_{L_0}
\mathrm{Int}_{3L_1} C(w_0)\setminus \cB_L(w_0)$. From the proof of part (i)
follows that there is some $v_0\in\mathrm{Cov}_{k+1}\subseteq
\mathrm{Int}_{3L_1} C(w_0)$ such that $w\in C(v_0)$. Since
$C(v_0)\subseteq \mathrm{Int}_{L_1} C(w_0)$ and $L_0\leq L_1$ we must have that
$w\in\partial_{3L_1} C(v_0)$. Also from the proof of part \emph{1} follows the
following: if $d(w,v_0)\leq L$ then there is a path from $w$ to
$\mathrm{Cov}_{k+2}\cap C(v_0)$ via $v_0$ which lies in $C(v_0)\subseteq \mathrm{Int}_{L_1}
C(w_0)$ and has a length less or equal to $2L$;  if $d(w,v_0)> L$ then there is a path from $w$ to
$\mathrm{Cov}_{k+2}\cap C(v_0)$ via $\partial C(v_0)$ which lies in $\mathrm{Int}_{L_1}
C(w_0)$ and has a length less or equal to $K$. That is, we have shown that
$\mathrm{Cov}_{k+2}$ can be reached from $w$ on a path running entirely through
$\mathrm{Int}_{L_1} C(w_0)$ of length at most $\max\{2L,K\}$. The remaining
case $w\in\partial_{L_0} \mathrm{Int}_{3L_1} C(w_0)\cap \cB_L(w_0)$ constitutes
finitely many cases and therefore choosing $K_1$ sufficiently large proves part \emph{2}.
\end{proof}

The motivation for the construction of the nested coverings $\mathrm{Cov}_{n}$
is that it allows us to trace the retracted walk's way to infinity: the limit
point $\overline{X}_{\infty}$ at ``infinity'' of the retracted walk is in
$1$-to-$1$-relation with a sequence of nested cones (defined by the coverings) that are visited all but finitely many times. To make this more precise, we  use the notation of  \textit{last entry times:} for $k\geq 0$ the $k$-th \textit{last entry time} is given by
\[
\e{k}:=\inf\bigl\lbrace m\in\N \,\bigl|\, \exists w\in \mathrm{Cov}_k: \,
\overline{X}_m\in \mathrm{Int}_{3L_1} C(w),\forall n\geq m:
\overline{X}_n\in \mathrm{Int}_{L_1} C(w) \bigr\rbrace.
\]
See also Figure \ref{fig:last exit}.
Furthermore, we define $\mathcal{R}_k:=w$ if
$\overline{X}_{\e{k}}\in C(w)$ with $w\in \mathrm{Cov}_k$. In other words,
$\mathcal{R}_k$ is the root of the cone whose $L_1$-interior is finally
entered by the retracted random walk $(\overline{X}_n)_{n\geq 0}$ at time $\e{k}$. 

\begin{figure}
\centering
 \begin{tikzpicture}[scale=2]
  \begin{scope}
  \draw (0,0) arc (-70:-30:5);
  \end{scope}
  \begin{scope}[yscale=-1]
  \draw (0,0) arc (-70:-46:5);
  \end{scope}
  \draw(0,0)node[left]{$w_{0}$};
  \fill(0,0) circle[radius=0.8pt, fill=black];

  \begin{scope}[xshift=1.8cm, yshift=0.1cm]
  \draw (0,0) arc (-55:-23.5:5);
  \fill(0,0) circle[radius=0.8pt, fill=black];
  \draw(0,0)node[left]{$w_{0}u$};
  \end{scope}
  
  \draw[-] (1.8,0.1) -- (4,0.1);

\begin{scope}[xshift=1.6cm, yshift=-0cm, yscale=-1]
  \draw (0,0) arc (-60:-40:5);
  \end{scope}
  \draw[-] (1.6,0) -- (4,0);


\begin{scope}[xshift=1cm, dashed]
  \draw (0,0) arc (-60:-25:5);
  \end{scope}
  \begin{scope}[yscale=-1, xshift=1cm, dashed]
  \draw (0,0) arc (-60:-40:5);
  \end{scope}

  \draw[dashed] (2.9,0.5) arc (-45:-21:5);
  
  \draw[-, dashed] (2.9,0.5) -- (4,0.5);

\draw[dotted] (2.3,0.25) arc (-50:-22:5);
  
  \draw[-, dotted] (2.3,0.25) -- (4,0.25);

\fill(3.45,1.02) circle[radius=0.8pt, fill=black];
\draw(3.45,1.02)node[right]{$\overline{X}_{{\mathbf{e}}_{k}}=w_{0}uy$};

\draw plot [smooth, tension=1.5, dashed] coordinates { (0,-1) (0.7,-0.3) (1,1)  (1.5,-0.5)   (2.5,1)  (3.41,1) (3.3,1.3) (4,1.5)};
\draw(0.7,0.9)node[above]{$\overline{X}_{n}$};

\filldraw[fill=white, draw=white] (0,2.2) rectangle (4.2,2.25);
\filldraw[fill=white, draw=white] (0,-1.15) rectangle (3,-1.2);

\end{tikzpicture}
\caption{Illustration of the last exit time $\mathbf{e}_{k}$, where $\mathcal{R}_k=w_0u\in\mathrm{Cov}_k$.}
\label{fig:last exit}
\end{figure}
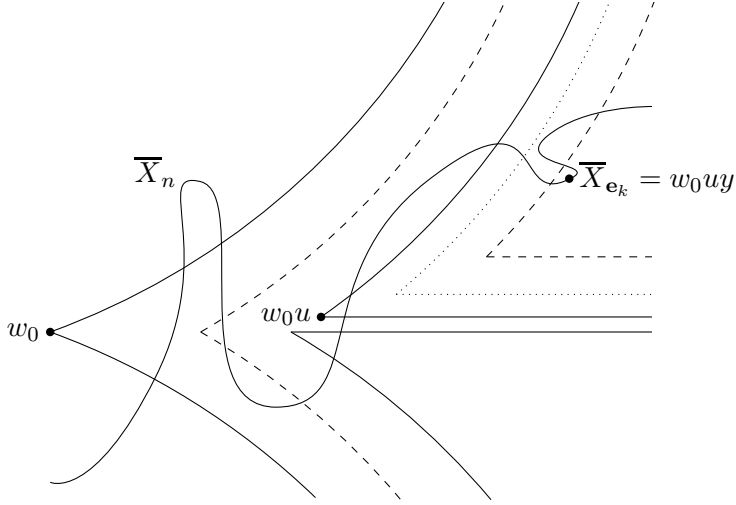

\begin{lem}\label{lem:last-entry-exponential-moments}
For all $k\in\N$ the last entry times $\e{k}$ are almost surely finite.  Moreover, the random variables $\e{0}$ and $\e{k}-\e{k-1}$ have uniform exponential moments, i.e., there exists some constants $\lambda_{\mathbf{e}}, K_{\mathbf{e}}>0$ 
such that $\E[\exp(\lambda_{\mathbf{e}}\e{0})]<K_{\mathbf{e}}$ and $\E[\exp(\lambda_{\mathbf{e}}(\e{k}-\e{k-1})]<K_{\mathbf{e}}$ for all $k\geq 1$.
\end{lem}
\begin{proof}[Sketch of proof]
We give an idea of the proof for $k=0$. The proof for $k\geq 1$ is similar. For the technical details we refer to the proof of  \cite[Theorem 5.5]{gilch-mueller-parkinson:14}. Let $L$ be the constant in Proposition \ref{prop:covering construction}.  Since $\varrho(\overline{P})<1$, there exists some constant $C$ such that
\[
\PP[\overline{X}_{n}\in\cB_L(e)]\leq C \varrho(\overline{P})^{n},
\]
e.g.~see Woess \cite[Lemma 8.1]{woess}.  In other words, the first exit time $\tau_{1}$ of the ball $\cB_L(e)$ has exponential moments. We wait a random time $\tau_{2}$ until the walk visits some point $v$ in $\partial_{L_{0}} \mathrm{Int}_{3L_1}(w)$ for some $w\notin \cB_L(e)$ . The time $\tau_{2}$ has again exponential moments. Arrived there, the walk has a positive probability of at least $\overline{p}_{esc}$ of staying in the cone $C(v)$ of the arrival point $v$, see Lemma \ref{lem:stayincone}. If it stays inside this cone we have that $\e{0}\leq \tau_{1}+\tau_{2}$.  If it does not stay inside the cone we wait some time $\tau_{3}$ until the walk leaves this cone for the first time. Note here that  this time has uniform (in the position of $\overline{X}_{\tau_{1}+\tau_{2}}$) exponential moments, see \cite[Lemma 5.12]{gilch-mueller-parkinson:14}. We then let $\tau_{4}$ be the time it takes after $\tau_{3}$ that the walk visits $\partial_{L_{0}} \mathrm{Int}_{3L_1}(w)$ for some $w\notin \cB_L(e)$ and so on.  The argument is now repeated until the first successful attempt to stay in one cone forever. Since  all the random variables $\tau_{i}$ have (uniform) exponential moments and a geometric sum of random variables with exponential moments has exponential moments, it follows hat $\e{0}$ has exponential moments.
\end{proof}

We define a new process $(\W_k)_{k\geq 0}$ on the state space $\mathcal{Z}:=\mfA_{R}\times  \bigcup_{\bT\in\mfA_R} \mathrm{Cov}(\bT) \times  W$ as follows:
the case $k=0$
plays a special role and we just set 
$$
\W_0:=\Bigl(T(\mathcal{R}_0),\mathcal{R}_0,\mathcal{R}_0^{-1}\overline{X}_{\e{0}}\Bigr);
$$
for $k\geq 1$, we set
$$
\W_k:=\Bigl(T(\mathcal{R}_k),\mathcal{R}_{k-1}^{-1}\mathcal{R}_k,\mathcal{R}_k^{-1}\overline{X}_{\e{k}}\Bigr).
$$
If
$\W_{k}=(\bT,u,y)$ then
there is some $w_0\in \mathrm{Cov}_{k-1}$ such that $u\in \mathrm{Cov}(w_0)$ with $T(w_0u)=\bT$
and $\overline{X}_{\e{k}}=w_0uy$; see Figure \ref{fig:last exit}.

The random variable $\W_0$ takes values in a set $\mathcal{W}_0:=\mathrm{supp}(\W_0)\subseteq
\mfA_{R}\times W \times W$, and the random
variables $\W_k$, $k\geq 1$, take values in a set $\mathcal{W}=\bigcup_{k\geq
  1} \mathrm{supp}(\W_k)\subseteq
\mathcal{Z}$. Thus, the sequence $(\W_k)_{k\geq 1}$ is a stochastic process  on
$\mathcal{W}$, which is induced by $(\overline{X}_n)_{n\geq 0}$.
\par
Observe that there is a 1-1-correspondence of the sequences
$\W_0,\dots,\W_k$ and $\overline{X}_{\e{0}},\dots,\overline{X}_{\e{k}}$: obviously,
for any given realisation of $\overline{X}_{\e{0}},\dots,\overline{X}_{\e{k}}$ we obtain unique
values for $\W_0,\dots,\W_k$; vice versa, for given values of
$\W_0=(\bT_0,u_0,y_0),\dots,\W_k=(\bT_k,u_k,y_k)$ we can successively describe which subcone is
entered one after the other one including the information which of the elements
in $W$ are the last entry time points, namely
$$
\mathcal{R}_k=\prod_{i=0}^k u_i,\quad \overline{X}_{\e{k}}= \mathcal{R}_ky_k.
$$
We define further generating functions: let be
$\w_1=(\bT_1,u_1,y_1),\w_2=(\bT_2,u_2,y_2)\in\mathcal{W}$ such that
$\Prob[\W_{k}=\w_1, \W_{k+1}=\w_2]>0$ for some $k\in\N$. Choose any  $w_0\in W$ with
$T(w_0)=\bT_1$ and define 
$$
\widehat{G}(\w_1,\w_2) := \sum_{n\geq 1}
\Prob_{w_0y_1}\left[
\begin{array}{c}
\forall m\leq n: \overline{X}_m\in \mathrm{Int}_{L_1} C(w_0), \\
\overline{X}_{n-1}\notin \mathrm{Int}_{3L_1} C(w_0u_2),
\overline{X}_n=w_0u_2y_2
\end{array}
\right].
$$
In words, the summands describe the probability that one walks from some
$w_0y_1$ to some $w_0u_2y_2$ in the following way: one starts at
$w_0y_1\in \mathrm{Int}_{3L_1}C(w_0)$, walks then inside $\mathrm{Int}_{L_1} C(w_0)$ to some
$w_0u_2y_2\in \mathrm{Int}_{3L_1} C(w_0u_2)\subset \mathrm{Int}_{L_1} C(w_0)$ such that the
step before arriving at $w_0u_2y_2$ is outside the $3L_1$-interior of the cone $C(w_0u_2)$.
The values $\widehat{G}(\w_1,\w_2)$ are well-defined since only paths inside the
$L_1$-interior of a cone of type $\bT_1$ are considered, implying that
the occuring probabilities depend only on the cone
type $\bT_1$ and \textit{not} on the specific choice of $w_0$, see Proposition \ref{prop:invariance}.

Let $\mathfrak{w}=(\bT,u,y)\in \mathcal{W}\cup\mathcal{W}_0$ and choose $w_0\in W$ such that
$T(w_0)=\bT$. Recall that $w_0y\in\mathrm{Int}_{3L_1} C(w_0)$ and therefore
$C(w_0y)\subseteq \mathrm{Int}_{L_1} C(w_0)$. Define
\begin{eqnarray*}
\xi(\w) & := & \Prob_{w_0y}\bigl[\forall n\geq 0: \overline{X}_n\in
\mathrm{Int}_{L_1} C(w_0)\bigr]\\
& = & 1-\sum_{n\geq 1} \Prob_{w_0y}\bigl[\overline{X}_n\in W\setminus \mathrm{Int}_{L_1}C(w_0),\forall m<n:
\overline{X}_n\in \mathrm{Int}_{L_1}C(w_0)\bigr].
\end{eqnarray*}
In other words, $\xi(\w)$ is the probability that the
$L_1$-interior of the cone $C(w_0)$ will never be exited when starting at $w_0y$.
Observe that the definition of $\xi(\w)$ is independent of the specific
choice of $w_0$ since we only consider paths \textit{inside} the $L_1$-interior
of the cone $C(w_0)$ (and
the first step into the $L_{1}$-boundary of the cone). An important fact is
that the values $\xi(\w)$ are uniformly positive:
\begin{lem}\label{lemma:xi>0}
There is a constant $C_0>0$ such that $\xi(\w)\geq C_0$ for all $\w\in\mathcal{W}\cup \mathcal{W}_0$.
\end{lem}
\begin{proof}
Let $\w=(\bT,u,y)\in \mathcal{W}\cup\mathcal{W}_0$ and $w_0\in W$ with
$T(w_0)=\bT$. Recall that $w_0y\in\mathrm{Int}_{3L_1} C(w_0)$ yielding
$C(w_0y)\subseteq \mathrm{Int}_{L_1} C(w_0)$. The claim follows now immediately
from Lemma \ref{lem:stayincone}.
\end{proof}
In the proof of the last lemma we use the fact that $X_{\mathbf{e}_k}\in\mathrm{Int}_{3L_1} C(\mathcal{R}_k)$. This explains why we force $X_{\mathbf{e}_k}$ to be in $\mathrm{Int}_{3L_1} C(\mathcal{R}_k)$ and not only in $\mathrm{Int}_{L_1} C(\mathcal{R}_k)$.
\begin{lem}\label{lem:W is ergodic}
The stochastic process $(\W_{k})_{k\geq 1}$ on the state space $\mathcal{W}$ is an irreducible, ergodic Markov chain with
transition probabilities
$$
q(\w_1,\w_2) =
\frac{\xi(\w_2)}{\xi(\w_1)} \widehat{G}\bigl(\w_1,\w_2\bigr),
$$
where $\w_1,\w_2\in\mathcal{W}$ with $\Prob[\W_k=\w_1,\W_{k+1}=\w_2]>0$ for some $k\in\N$.
\end{lem}
\begin{proof}
Let be $\w_1,\dots,\w_{k+1}\in\mathcal{W}$ with
$\Prob[\W_1=\w_1,\dots,\W_{k+1}=\w_{k+1}]>0$.  For any given
$\w_0=(\bT,u,y)\in\mathcal{W}_0$,
we define 
$$
\widehat{G}(e,\w_0) := \sum_{n\geq 1} 
\Prob\bigl[\overline{X}_{n-1}\notin \mathrm{Int}_{3L_1} C(u),\overline{X}_n=uy \bigr].
$$
Write $\w_i=(\mathbf{x}_i,u_i,y_i)$ and set $v_0:=u_0y_0$ and $v_j:=u_0u_1\dots
u_jy_j$ for $j\geq 1$.
Then by definition of
$\widehat{G}(\cdot,\cdot)$ and $\xi(\cdot)$ and due to the $1$-$1$-relation of
$\W_0,\dots,\W_k$ and $\overline{X}_{\e{0}},\dots,\overline{X}_{\e{k}}$ we get:
\begin{align*}
&\hspace{-1.5cm} \Prob\bigl[\W_1=\w_1,\dots,\W_{k}=\w_{k}\bigr] \\
& =   
\sum_{\w_0\in\mathcal{W}_0}
\Prob\bigl[\W_0=\w_0,\W_1=\w_1,\dots,\W_{k}=\w_{k}\bigr]\\
& =   
\sum_{\w_0\in\mathcal{W}_0}
\Prob\bigl[\overline{X}_{\mathbf{e}_0}=v_0,\overline{X}_{\mathbf{e}_1}=v_1,\dots,\overline{X}_{\mathbf{e}_{k}}=v_{k}\bigr]\\
&= \sum_{\w_0\in\mathcal{W}_0} \Prob\left[
\begin{array}{c}
\overline{X}_{\mathbf{e}_0-1}\notin \mathrm{Int}_{3L_1} C(u_0), \overline{X}_{\mathbf{e}_0}=u_0y_0,\\
\forall j\in\{1,\dots,k\} \forall n_j\geq \e{j-1}:
\overline{X}_{n_j}\in\mathrm{Int}_{L_1} C(u_0\dots u_{j-1}),\\
\overline{X}_{\e{j}-1}\notin \mathrm{Int}_{3L_1} C(u_0\dots u_{j}),
\overline{X}_{\e{j}}=u_0u_1\dots u_jy_{j},\\
\forall n\geq \e{k}: \overline{X}_{n}\in\mathrm{Int}_{L_1} C(u_0\dots u_{k})
\end{array}
\right] \\
& = \sum_{\w_0\in\mathcal{W}_0} \widehat{G}(e,\w_0) \cdot \widehat{G}(\w_0,\w_1)\cdot  \widehat{G}(\w_1,\w_2)\cdot \ldots \cdot \widehat{G}(\w_k,\w_{k})\cdot \xi(\w_{k}).
\end{align*}
The last equation arises by splitting up the paths in the event
$[X_{\mathbf{e}_0}=x_0,\dots,X_{\mathbf{e}_{k}}=x_{k}\bigr]$
with respect to~their part between $X_{\mathbf{e}_{j-1}}$ and $X_{\mathbf{e}_{j}}$,
which is described by $\widehat{G}(\w_{j-1},\w_j)$.
Hence,
$$
\Prob\bigl[\W_{k+1}=\w_{k+1}\,\bigl|\, \W_1=\w_1,\dots,\W_k=\w_k\bigr]=
\frac{\xi(\w_{k+1})}{\xi(\w_{k})}\widehat{G}(\w_k,\w_{k+1}).
$$
Irreducibility follows from the construction of the coverings that each
covering has subcones of all types together with~(\ref{eq:support}).
\par
In order to show ergodicity of the process we prove positive recurrence and aperiodicity. First, we show that the process is positive recurrent. Due to irreducibility it is sufficient to show that $\w=(\bT,u,y)\in\mathcal{W}$ with $|u|\leq L$ can be reached from any $\w_0=(\bT_0,u_0,y_0)\in\mathcal{W}$ in four steps with positive probability bounded away from zero. For this purpose, take any $k\in\mathbb{N}$ and $x_0\in \mathrm{Cov}_k$ with $T(x_0)=\bT_0$. Then $x_0y_0\in\partial_{L_0}\mathrm{Int}_{3L_1}C(x_0)$ and, by Lemma \ref{lem:cover-dense}, there is some $x_2\in\mathrm{Cov}_{k+2}\cap C(x_0)$, which can be reached from $x_0y_0$ on a path inside $\mathrm{Int}_{L_1} C(x_0)$ of length at most $K_1$. Take any $u_3\in\mathrm{Cov}(T(x_2))$ such that $\ell(u_3)\leq L$ and $u\in \mathrm{Cov}(T(x_2u_3))$. Then one can walk from $x_0y_0$ via $x_2$, $x_2u_3$ and  $x_2u_3u$ to $x_2u_3uy$ on a path inside $\mathrm{Int}_{L_1}C(x_0)$. This yields $\Prob[\W_4=\w \mid |\W_0=\w_0]\geq \varepsilon_0^{K_1+2L+\ell(y)}$, which proves positive recurrence of $(\W_k)_{k\in\N}$ by a standard geometric distribution argument.
\par

\par
Aperiodicity is obtained as follows: if we start at some $\w\in\mathcal{W}$ then we can
  come back to $\w$ in four steps with positive probability. Take any $\hat
  \w\in\mathcal{W}$ with $q(\w,\hat \w)>0$. Analogously, starting at $\hat \w$ we
  can reach $\w$ in four steps with positive probability, yielding that we can also
  reach $\w$ with positive probability in five steps when starting at $\w$. This
  yields aperiodicity, and thus ergodicity.

\end{proof}

\subsection{Entropy of a hidden Markov chain related to the last entry time process}
\label{sec:entropy-HMC}

In this subsection we introduce a hidden Markov chain and consider its
asymptotic entropy, which will be linked with the entropy of $(\overline{X}_n)_{n\geq 0}$ in the next subsection.
\par
We define the function $\Phi:(\mathcal{W}\cup\mathcal{W}_0)\times \mathcal{W} \to \mfA_R\times
\mfA_R\times \bigcup_{\bT\in\mfA_R} \mathrm{Cov}(\bT)$ by
$$
\Phi\bigl((\bT_1,u_1,y_1),(\bT_2,u_2,y_2)\bigr)
:=\bigl(\bT_{1},\bT_2,u_2\bigr).
$$
This leads to the \textit{hidden Markov chain} $(\Y_k)_{k\geq 1}$ defined by
$$
\Y_k:=\Phi(\W_{k-1},\W_k).
$$
The random variables $\Y_k$, $k\in\mathbb{N}$, give information about the random walk's
way to infinity; they describe which cones of the covering are finally entered without keeping the information of the exact last entry points of the cones.  Given the value of
$\bT_{1}$ we know the relative position of the next subcone
described by
$(\bT_2,u_2)$. While
$(\W_{k-1},\W_k)_{k\geq 1}$ is a Markov chain, the
process $(\Y_k)_{k\geq 1}$ is in general not Markovian. 
\par
%
%
%
Recall that the Markov process $(\W_k)_{k\geq 1}$ is positive recurrent,
and therefore there exists a stationary probability measure $\nu$. Thus,
$(\W_k)_{k\geq 1}$ is asymptotically mean stationary and so is the hidden Markov chain $(\Y_k)_{k\geq 1}$.

We introduce an additional random variable $\Y_0$: if $\W_0=(\bT,u,y)\in\mathcal{W}_0$ then we set $\Y_0:=u$. In other words,
$\Y_0$ describes the  root $\mathcal{R}_0$ of the cone associated with the initial
last entry point $\overline{X}_{\e{0}}$, without keeping the information of the
exact location of $\overline{X}_{\e{0}}\in C(\mathcal{R}_0)$. This information of $\Y_0$ is, in addition to the values
of $\Y_1,\Y_2,\dots$, needed for tracking the random walk's route to infinity, since $\Y_1=(\mathbf{T}_0,\mathbf{T}_1,u_1)$
determines only the type $\bT_0$ of $C(\mathcal{R}_0)\in \mathrm{Cov}_0$, but
there may be several different cones of type $\bT_0$ in $\mathrm{Cov}_0$.

A generalized version of the famous Shannon-McMillan-Breiman theorem states
that then there exists a non-negative constant $H(\Y)\in[0,\infty]$ such that 

\begin{equation}\label{equ:shannon}
H(\Y)=\lim_{n\to\infty} -\frac{1}{n} \log \Prob\bigl[\Y_0=\mathfrak{y}_0,\dots,\Y_n=\mathfrak{y}_n\bigr]
\end{equation}
for almost every realisation $(\mathfrak{y}_0,\mathfrak{y}_1,\dots)$ of the process
$(\Y_k)_{k\geq 1}$; see Algoet and Cover \cite[Theorem~4]{AlCo:88}.
The number $H(\Y)$ is called the \textit{asymptotic entropy} of the process
$(\Y_k)_{k\geq 1}$. We will see later that $H(\Y)$ is finite and prove now that $H(\Y)$ is strictly positive.
\begin{lem}
$H(\Y)>0$.
\end{lem}
\begin{proof}
Take any $\w_1,\w_2\in\mathcal{W}$ with
$\Prob[\W_1=\w_1,\W_2=\w_2]>0$.
The values $\w_1,\w_2$ determine the value of $\Y_2$ uniquely. Since $W$ grows
exponentially, the subgraph $\mfA_R$ is not a circle. Therefore, the
construction of coverings ensures that
 there are at least two elements $\w',\w''\in\mathcal{W}$ with
$q(\w_2,\w')>0$, $q(\w_2,\w'')>0$ and $\Phi(\w_2,\w') \neq
\Phi(\w_2,\w'')$. Then:
\begin{align*}
&\Prob\bigl[\Y_3=\Phi(\w_2,\w') \,\bigl|\,
\W_1=\w_1,\W_2=\w_2\bigr]\geq q(\w_2,\w')>0,
\\
& \Prob\bigl[\Y_3=\Phi(\w_2,\w'') \,\bigl|\,
\W_1=\w_1,\W_2=\w_2\bigr]\geq q(\w_2,\w'')>0.
\end{align*}
Thus, $P(\w' \mid \w_1,\w_2):=\Prob\bigl[\Y_3=\Phi(\w_2,\w') \,\bigl|\,
\W_1=\w_1,\W_2=\w_2)\bigr]<1$. 
\par
Recall that the conditional entropy of discrete random variables $A_1$ and
$A_2$ on a state space $\mathcal{S}_0$ is defined as
$$
H(A_{2} \mid A_{1}) := -\sum_{a_{1},a_2\in\mathcal{S}_0}
\Prob[A_{1}=a_1,A_{2}=a_2]\log \Prob\bigl[A_{2}=a_2 \mid A_{1}=a_{1}\bigr]
$$
and analogously for more variables.
From Cover and Thomas \cite[Theorem 4.5.1]{cover-thomas} follows then
\begin{eqnarray*}
H(\Y) &\geq &   H\bigl(\Y_3 \bigl|\mathbf{W}_1,\mathbf{W}_2,\Y_2\bigr)\\
&\geq&  \nu(\w_1)q(\w_1,\w_2) P(\w' \mid \w_1,\w_2)\cdot \log P(\w' \mid \w_1,\w_2) >0,
\end{eqnarray*}
where $\nu$ is the invariant probability measure of $(\W_k)_{k\geq 1}$.
This yields the claim.
\end{proof}

\subsection{Rate of escape with respect to the Green distance}
\label{sec:entropy-retracted}

We define a new ``length function'' such that $H(\Y)$ becomes the rate of
escape with respect to this length function. Let be $w_0\in W$ and consider the cone $C(w_0)$ rooted at $w_0$. We define 
$$
l(w_0) := -\log \sum_{v\in\partial_{L_0} \mathrm{Int}_{3L_1} C(w_0)} \overline{G}(e,v).
$$
The next lemma shows that $l(x_0)$ is well-defined.
\begin{lem}\label{lemma:sum-L-finite}
Let be $w_{0}\in W$. Then
$$
\sum_{v\in\partial_{L_0} \mathrm{Int}_{3L_1} C_{}(w_0)} \overline{G}(e,v)<\infty.
$$
\end{lem}
\begin{proof}
We use the fact that the boundary of a cone only grows linearly (see Lemma
\ref{lem:boundary-rays}) and that
\begin{equation}\label{equ:boundary-inclusion}
\partial_{L_{0}} \mathrm{Int}_{3L_1} C(w_0) \cap \cB_{n}(w_{0})\subseteq \bigcup_{u\in\partial C_{}(w_0)  \cap \cB_{n+3L_1+L_0}(w_{0})} \cB_{3L_{1}+L_{0}}(u).
\end{equation}
Using the fact that the Green functions decay exponentially (see (\ref{equ:Green-estimate})) we get that
\begin{align*}
\sum_{v\in\partial_{L_0} \mathrm{Int}_{3L_1} C(w_0)} \overline{G}(e,v) &\leq 
C\cdot \sum_{v\in\partial_{L_0} \mathrm{Int}_{3L_1} C(w_0)}
\lambda^{d(e,v)}\\ 
&= C\cdot \sum_{v\in\partial_{L_0} \mathrm{Int}_{3L_1} C(w_0)}
\lambda^{d(e,w_{0})+d(w_{0},v)} \\
&\leq  2C|\cB_{3L_1+L_0}(e)| \lambda^{d(e,w_{0})} \sum_{n\geq 0} (n+3L_1+L_0) \lambda^n<\infty.
\end{align*}
\end{proof}

We introduce some further notation and generating functions. For any $\bT\in\mfA_R$,
we write $v\in\partial_{L_0}\mathrm{Int}_{3L_1} C(\bT)$ if, for any
(or equivalently, for all) $w\in W$ with $\T(w)=\bT$, we have
$wv\in\partial_{L_0}\mathrm{Int}_{3L_1} C(w)$. For any $\bT\in\mfA_R$ and $v\in \partial_{L_0}\mathrm{Int}_{3L_1} \bT$, define
$$
\widetilde{G}(v,\partial_{L_0}\mathrm{Int}_{3L_1} \bT):=\sum_{n\geq 0}
\Prob_{wv}\bigl[\overline{X}_n\in \partial_{L_0}\mathrm{Int}_{3L_1}C(w),\forall
m<n: \overline{X}_m\in \mathrm{Int}_{L_1}C(w) \bigr],
$$
where $w\in W$ with $\T(w)=\bT$. This definition is independent from the
specific choice of $w$ since we consider only paths inside the $L_1$-interior of a cone of type
$\bT$. Similar to Lemma \ref{lemma:sum-L-finite} one shows that the values
  $\widetilde{G}(v,\partial_{L_0}\mathrm{Int}_{3L_1} \bT)$ are uniformly
  bounded.
\begin{lemma}\label{lem:tilde-G-bounded}
There is some constant $\tilde C>0$ such that
$\widetilde{G}(v,\partial_{L_0}\mathrm{Int}_{3L_1} \bT)<\tilde C$ for all $\bT\in\mfA_R$ and $v\in \partial_{L_0}\mathrm{Int}_{3L_1} \bT$.
\end{lemma}

\begin{prop}\label{prop:hat-l-convergence}
$$
\lim_{k\to\infty} \frac{l(\mathcal{R}_k)}{k}=H(\Y)\quad \textrm{ almost surely}. 
$$
\end{prop}
\begin{proof}
Consider a realization of the random walk $(\overline{X}_n)_{n\geq 0}$, where the instances of $\W_k$ are given
by $(\bT_k,u_k,v_k)$.  We set $r_j:=u_0\dots u_j$ for $j\geq 0$. Thus, $C(r_j)$
is the cone associated with $\overline{X}_{\e{j}}$,
\emph{i.e.}, $\overline{X}_{\e{j}}$ is in
$\partial_{L_0}\mathrm{Int}_{3L_1}C(r_j)$.
\par
Each path from $e$ to
$\overline{X}_{\e{k}}$ must successively pass through the boundaries
$\partial_{L_0}\mathrm{Int}_{3L_1} C(r_j)$ for every $j=0,\dots,k$. By conditioning on the last entry points of these paths 
and using the fact that $\widetilde{G}(v_k,\partial_{L_0}\mathrm{Int}_{3L_1} \bT_k)$ is uniformly bounded
we get for almost every realisation $(r_0,r_1,\dots)$ of
$(\mathcal{R}_{k})_{k\geq 0}$ (which are implied by the realizations
$(\bT_k,u_k,v_k)$ of $\W_k$)
that $l(r_{k})$ equals
$$
\sum_{\substack{\bar v_0,\bar
    v_1,\dots,\bar v_k\in W:\\
    \bar v_i\in\partial_{L_0}\mathrm{Int}_{3L_1} C(r_i)}} \widehat{G}\bigl(e,(\bT_0,u_0,\bar
v_0)\bigr) \prod_{j=1}^k \widehat{G}\bigl((\bT_{j-1},u_{j-1},\bar
v_{j-1}),(\bT_j,u_j,\bar v_j)\bigr)\cdot \widetilde{G}(\bar
v_k,\partial_{L_0}\mathrm{Int}_{3L_1} \bT_k).
$$
We note that, by definition, we have $\Prob[\W_0=\w_0]=\widehat{G}(e,\w_0)\xi(\w_0)$ for $\w_0\in\mathcal{W}_0$. Recall that $1\geq \xi(\w)\geq C_0>0$
for all $\w\in\mathcal{W}$ by Lemma \ref{lemma:xi>0}. Together with
Lemma \ref{lem:tilde-G-bounded} we get the following convergence for almost
every realization $(r_k)_{k\geq 0}$ of $(\mathcal{R}_{k})_{k\geq 0}$:
\begin{eqnarray*}
&&\lim_{k\to\infty} \frac{l(r_k)}{k}\\
&=& \lim_{k\to\infty} -\frac{1}{k} \log \sum_{\substack{\bar v_0,\bar
    v_1,\dots,\bar v_k\in W:\\
    \bar v_i\in\partial_{L_0}\mathrm{Int}_{3L_1} C(r_i)}} \widehat{G}\bigl(e,(\bT_0,u_0,\bar
v_0)\bigr) \xi(\bT_0,u_0,\bar v_0)\cdot\\
&&\quad\quad\quad\quad\quad\quad\quad\quad  \cdot \prod_{j=1}^k \frac{\xi(\bT_{j},u_{j},\bar
  v_{j})}{\xi(\bT_{j-1},u_{j-1},\bar v_{j-1})} \widehat{G}\bigl((\bT_{j-1},u_{j-1},\bar v_{j-1}),(\bT_j,u_j,\bar v_j)\bigr)\\
&=&\lim_{k\to\infty} -\frac{1}{k} \log \sum_{\substack{\bar v_0,\bar
    v_1,\dots,\bar v_k\in W:\\
    \bar v_i\in\partial_{L_0}\mathrm{Int}_{3L_1} C(r_i)}} \Prob[\W_{0}=(\bT_0,u_0,\bar v_0)]
\prod_{j=1}^k q\bigl((\bT_{j-1},u_{j-1},\bar v_{j-1}),(\bT_j,u_j,\bar v_j)\bigr)\\
&=&\lim_{k\to\infty} -\frac{1}{k} \log \Prob\bigl[\Y_0=u_0,
\Y_1=(\bT_0,\bT_1,u_1),\dots,\Y_k=(\bT_{k-1},\bT_k,u_k)\bigr] =H(\Y). 
\end{eqnarray*}
%
%
\end{proof}

We write $\mathbb{E}_\nu$ if we replace the
original initial distribution of $\W_1$ by $\nu$. An application of the ergodic theorem for
positive recurrent Markov chains together with Lemma \ref{lem:last-entry-exponential-moments} yields the following lemma.

\begin{lemma}\label{lem:e_k/k-convergence}
$$
\frac{\e{k}}{k}  \quad \xrightarrow{k\to\infty} \quad \E_{\nu}[\e{2}-\e{1}]<\infty
\quad \textrm{ almost surely}. 
$$
\end{lemma}
We will see that it is sufficient to consider only the last entry
times of the retracted walk in order to prove existence of the asymptotic
entropy. For this purpose we use the following lemma that is a consequence of the fact that $\bigl(d(\mathcal{R}_k,\overline{X}_{\e{k}})\bigr)_{k\geq 1}$ is a functional of the ergodic Markov chain $(\W_{k})_{k\geq 1}$.
\begin{lem}\label{lem:root-last-entry-distance}
$$
\frac{d(\mathcal{R}_k,\overline{X}_{\e{k}})}{k} \quad \xrightarrow{k\to\infty} \quad 0 \quad \textrm{ almost surely}.
$$
\end{lem}
Due to (\ref{equ:speed}) and $\ell(X_{\e{k}})=\ell(\overline{X}_{\e{k}})$ we have
\begin{equation}\label{equ:retracted-drift}
\lim_{k\to\infty} \frac{d(e,\overline{X}_{\e{k}})}{\e{k}} =
\lim_{k\to\infty} \frac{d(e,\overline{X}_{\e{k}})}{k}\frac{k}{\e{k}} =\mathfrak{v}\quad \textrm{ almost surely}.
\end{equation}
Since $d(e,\mathcal{R}_k)=\sum_{i=0}^k d(\mathcal{R}_{i-1},\mathcal{R}_i)=\sum_{i=0}^k |u_i|$, where $\mathcal{R}_{-1}:=e$ and $\mathbf{W}_i=(\bT_i,u_i,y_i)$, another application of the ergodic theorem yields almost surely $d(e,\mathcal{R}_k)/k \to \E_{\nu}[d(\mathcal{R}_2, \mathcal{R}_1)]$ as $k\to\infty$.
With Lemmata \ref{lem:e_k/k-convergence} and \ref{lem:root-last-entry-distance} and using  (\ref{equ:retracted-drift}) we can now express the drift $\mathtt{v}$ of $(\overline{X}_n)_{n\geq 0}$  as
\begin{equation}\label{equ:drift-formula}
\mathtt{v}=\frac{\E_{\nu}[d(\mathcal{R}_2, \mathcal{R}_1)]}{\E_{\nu}[\e{2}-\e{1}]}.
\end{equation}

In the next step we express $H(\Y)$ in terms of the  Green distance.
\begin{prop}\label{prop:l-convergence}
$$
 \frac{\ell_{G}(\overline{X}_{\e{k}})}{k} \quad \xrightarrow{k\to\infty} \quad H(\Y)\quad \textrm{ almost surely}.
$$
\end{prop}
\begin{proof} On the one hand side we have that
\begin{equation}\label{equ:ell-l-inequ-1}
\exp(-\ell_{G}(\overline{X}_{\e{k}}))=
\frac{\overline{G}(e,\overline{X}_{\e{k}})}{\overline{G}(e,e)} \leq \sum_{v\in\partial_{L_0}\mathrm{Int}_{3L_1} C(\mathcal{R}_k)} \frac{\overline{G}(e,v)}{\overline{G}(e,e)}=\exp(-l(\mathcal{R}_k)).
\end{equation}
On the other hand side we have
\begin{equation}\label{equ:G-inequality-in-proof-of-l-convergence}
\overline{G}(e,\mathcal{R}_k)\cdot \varepsilon_0^{d(\mathcal{R}_k,\overline{X}_{\e{k}})} \leq \overline{G}(e,\overline{X}_{\e{k}});
\end{equation}
recall that $\varepsilon_0>0$ is the minimal single-step transition probability of the
random walk $(\overline{X}_n)_{n\geq 0}$. As in the proof of Lemma \ref{lemma:sum-L-finite} we have that $ \partial_{L_0}\mathrm{Int}_{3L_1} C(\mathcal{R}_k)\cap \cB_n(\mathcal{R}_k)$ grows linearly in $n$. 

Recall Ancona's Inequality (see e.g.~\cite[Theorem 27.12]{woess}): there is some $C>0$ such that 
$$
\overline{G}(u,v) \leq \widetilde{C} \cdot \overline{G}(u,w)\cdot \overline{G}(w,v) 
$$
for all $u,v,w\in W$ with $w$ being on a geodesic from $u$ to $v$. Since
$\mathcal{R}_k$ is on a geodesic from $e$ to any $w\in C(\mathcal{R}_k)$, we get with (\ref{equ:Green-estimate}) and (\ref{equ:boundary-inclusion}):
\begin{eqnarray}
\exp(-l(\mathcal{R}_k))&=&\sum_{w\in\partial_{L_0}\mathrm{Int}_{3L_1} C(\mathcal{R}_k)} \frac{\overline{G}(e,w)}{\overline{G}(e,e)} \\
& \leq & \frac{\widetilde{C}}{\overline{G}(e,e)} \cdot  \sum_{w\in\partial_{L_0}\mathrm{Int}_{3L_1}
  C(\mathcal{R}_k)} \overline{G}(e,\mathcal{R}_k) \overline{G}(\mathcal{R}_k,y)\nonumber\\
&\leq &  \frac{\widetilde{C} \cdot \overline{G}(e,\mathcal{R}_k)}{\overline{G}(e,e)} \sum_{w\in\partial_{L_0}\mathrm{Int}_{3L_1}
  C(\mathcal{R}_k)}C\cdot \lambda^{d(\mathcal{R}_k,w)}\nonumber\\
&\leq & \frac{ \widetilde{C}C}{\varepsilon_0^{d(\mathcal{R}_k,\overline{X}_{\e{k}})}} \cdot
\frac{\overline{G}(e,\mathcal{R}_k)}{\overline{G}(e,e)}\cdot \varepsilon_0^{d(\mathcal{R}_k,\overline{X}_{\e{k}})}\cdot
\sum_{n\geq 0} 2(n+3L_1+L_0)|\cB_{3L_1+L_0}(e)|\lambda^n  \nonumber\\
&\leq& \frac{\widetilde{C}C}{\varepsilon_0^{d(\mathcal{R}_k,\overline{X}_{\e{k}})}} \frac{\overline{G}(e,\overline{X}_{\e{k}})}{\overline{G}(e,e)}\cdot |\cB_{3L_1+L_0}(e)| \cdot \left( \frac{6L_1+2L_0}{1-\lambda}+\frac{2\lambda}{(1-\lambda)^2}\right).,\label{equ:ell-l-inequ-2}
\end{eqnarray}
The last inequality follows from (\ref{equ:G-inequality-in-proof-of-l-convergence}).
Together with Proposition \ref{prop:hat-l-convergence} and Lemma
\ref{lem:root-last-entry-distance} we get the claim since (\ref{equ:ell-l-inequ-1}) and (\ref{equ:ell-l-inequ-2}) allow to
compare $l(\mathcal{R}_k)$ with $\ell_{G}(\overline{X}_{\e{k}})$.
\end{proof}
Now we can prove existence of the rate of escape with respect to the Green distance of the
retracted walk:
\begin{thm}\label{thm:l-entropy} We have
$$
 \frac{\ell_{G}(\overline{X}_n)}{n} \quad \xrightarrow{k\to\infty} \quad \frac{H(\Y)}{\E_{\nu}[\e{2}-\e{1}]}:=h_{G}>0.
\quad \textrm{ almost surely,}
$$
\end{thm}
\begin{proof}
For $n\in\N$, define
\begin{equation}\label{equ:def-k(n)}
\mathbf{k}(n)  :=  \max \{k\in\N \mid \e{k}\leq n\}.
\end{equation}
With this notation we have
$$
\lim_{n\to\infty} \frac{\ell_{G}(\overline{X}_n)}{n} = \lim_{n\to\infty}
\frac{\ell_{G}(\overline{X}_n)-\ell_{G}(\overline{X}_{\e{\mathbf{k}(n)}})}{n}
+\frac{\ell_{G}(\overline{X}_{\e{\mathbf{k}(n)}})}{\mathbf{k}(n)}\frac{\mathbf{k}(n)}{\e{\mathbf{k}(n)}}\frac{\e{\mathbf{k}(n)}}{n}.
$$
By Proposition \ref{prop:l-convergence},
$\ell_{G}(\overline{X}_{\e{\mathbf{k}(n)}})/\mathbf{k}(n)$ tends almost surely to
$H(\Y)$. By Lemma \ref{lem:e_k/k-convergence},
$\e{\mathbf{k}(n)}/\mathbf{k}(n)$ tends almost surely to $\E_{\nu}[\e{2}-\e{1}]$ as
$n\to\infty$. Since
$$
1\leq \frac{n}{\e{{\mathbf{k}(n)}}}\leq
\frac{\mathbf{e}_{\mathbf{k}(n)+1}}{\mathbf{e}_{\mathbf{k}(n)}}=\frac{\mathbf{e}_{\mathbf{k}(n)+1}}{\mathbf{k}(n)+1}\frac{\mathbf{k}(n)+1}{\mathbf{e}_{\mathbf{k}(n)}}\xrightarrow{n\to\infty}
1 \ \textrm{ almost surely}
$$
we have $\lim_{n\to\infty}\e{{\mathbf{k}(n)}}/n=1$ almost surely.
\par
It remains to investigate the difference $\ell_{G}(\overline{X}_n)-\ell_{G}(\overline{X}_{\e{\mathbf{k}(n)}})$.
Obviously, one can walk with positive probability in $n-\e{\mathbf{k}(n)}$ steps
from $\overline{X}_{\e{\mathbf{k}(n)}}$ to $\overline{X}_n$ and vice versa; this probability is at
least $\varepsilon_0^{n-\e{\mathbf{k}(n)}}$, where $\varepsilon_0>0$ is the
minimal single-step transition probability of
$(\overline{X}_n)_{n\geq 0}$. Hence, we have the following estimates:
$$
G(e,\overline{X}_n)\cdot \varepsilon_0^{n-\e{\mathbf{k}(n)}}  \leq  
G(e,\overline{X}_{\e{\mathbf{k}(n)}}),\quad\quad
G(e,\overline{X}_{\e{\mathbf{k}(n)}}) \cdot \varepsilon_0^{n-\e{\mathbf{k}(n)}}   \leq  G(e,\overline{X}_n).
$$
Since $\lim_{n\to\infty}\e{{\mathbf{k}(n)}}/n=1$ we get
\begin{equation}\label{equ:n-e-convergence}
\frac{n-\e{\mathbf{k}(n)}}{n}\xrightarrow{n\to\infty} 0\quad \textrm{ almost surely.}
\end{equation}
Thus, the quotient $\bigl(\ell_{G}(\overline{X}_n)-\ell_{G}(\overline{X}_{\e{\mathbf{k}(n)}})\bigr)/n$ tends to
zero almost surely. This finishes the proof.
\end{proof}
By (\ref{equ:drift-formula}) we can give another formula for $h_G$ in terms of the rate of escape $\mathtt{v}$ of $(X_n)_{n\geq 0}$:
\begin{equation}\label{equ:h_G}
h_G = \frac{H(\Y) \cdot \mathtt{v}}{ \E_{\nu}[d(\cR_{2},\cR_{1})]}.
\end{equation}

\subsection{The entropy of the retracted walk}
\label{subsec:entropy-retracted-walk}

We follow the reasoning of \cite{gilch:11} for the proof of existence of the entropy.
First, we remark that due to irreducibility of the retracted walk and by
\cite[Proposition 4.6, Corollary 4.9]{gilch-mueller-parkinson:14} we have a unique radius of convergence $R>1$ of $G(u,v|z)$ for all $u,v\in W$.
In the following let be $r\in[1,R)$ and recall that $\varepsilon_0$ is
the minimal positive single-step transition probability of the retracted walk. 
The following technical lemma that is an adaptation of \cite[Lemma 3.6]{gilch:11}  will be used in the proof of the next theorem.

\begin{lemma}\label{lemma:sum-bounds}
For $n\in\N$, consider the function $f_n:W\to\R$ defined by
$$
f_n(w):=\begin{cases}
-\frac{1}{n}\log \sum_{m=0}^{n^2} \overline{p}^{(m)}(e,w), & \textrm{if }
\overline{p}^{(n)}(e,w)>0,\\
0, & \textrm{otherwise.}
\end{cases}
$$
Then there are constants $d$ and $D$ such that $d\leq f_n(w)\leq D$ for all
$n\in\N$ and $w\in W$.
\end{lemma}

\begin{thm}\label{thm:L-limit}
The entropy $\overline{h}$ exists
and equals $h_G$. In particular, the asymptotic  entropy is the rate of escape with respect to the Green distance.
\end{thm}
\begin{proof}
We can rewrite $h_G$ as
$$
h_G =  \int h_G\,d\Prob =
\int \lim_{n\to\infty} \frac{1}{n}\ell_{G}(\overline{X}_n)\,d\Prob\\
=
\int \lim_{n\to\infty} -\frac{1}{n} \log  \overline{G}\bigl(e,\overline{X}_n\bigl|1\bigr)\,d\Prob.
$$
Since
$$
\overline{G}(e,\overline{X}_n|1) = \sum_{m\geq 0} \overline{p}^{(m)}(e,\overline{X}_n) \geq \overline{p}^{(n)}(e,\overline{X}_n) = \overline{\pi}_n(\overline{X}_n),
$$
we have
\begin{equation}\label{equ:liminf-h}
h_G \leq \int \liminf_{n\to\infty} -\frac{1}{n} \log \overline{\pi}_n\bigl(\overline{X}_n\bigr)\,d\Prob.
\end{equation}
The next aim is to prove $\limsup_{n\to\infty} -\frac{1}{n}\E\bigl[\log
\overline{\pi}_n(\overline{X}_n)\bigr] \leq h_G$. 
We  decompose $\sum_{m\geq 0} \overline{p}^{(m)}(e,\overline{X}_n)$ as
$$
\sum_{m=0}^{n^2} \overline{p}^{(m)}(e,\overline{X}_n) \textrm{ and } b_n:=\sum_{m\geq n^2+1}
\overline{p}^{(m)}(e,\overline{X}_n).
$$
Let us  control the sequence $(b_{n})_{n\geq 1}$. There exists a constant
  $A$ such that, for any $N\in\mathbb{N}$, 
\begin{equation*}
\sum_{w\in \cB_N(e)} \overline{p}^{(m)}(e,w)\leq |S|^{N} A^{N} R^{-m},
\end{equation*}
see \cite[Lemma 8.1]{woess}. Hence,  there exists a constant $C_1$  such that
for all $m,n\in\N_0$  we have 
$$
\overline{p}^{(m)}(e,\overline{X}_n)
\leq \sum_{w\in \cB_{nL_0}(e)}\overline{p}^{(m)}(e,w) \leq C_1^{n} \cdot R^{-m}\quad \textrm{almost surely.}
$$
Hence, 
$$
b_n\leq \sum_{m\geq n^2+1} C_1^n \cdot  R^{-m}
=   C_1^n   \cdot
\frac{R^{-n^2-1}}{1-R^{-1}}. 
$$
Therefore, $b_n$ decays faster than any geometric sequence and hence
$$
h_G=\lim_{n\to\infty}  -\frac{1}{n} \log \sum_{m=0}^{n^2}
\overline{p}^{(m)}\bigl(e,\overline{X}_n\bigr)\quad \textrm{almost surely.}
$$
By Lemma \ref{lemma:sum-bounds}, 
we may apply the dominated convergence theorem and get:
\begin{eqnarray*}
h_G =\int h_G \,d\Prob 
& = &  \int \lim_{n\to\infty} -\frac{1}{n} \log \sum_{m=0}^{n^2}
\overline{p}^{(m)}(e,\overline{X}_n)\, d\Prob \\
&=&
\lim_{n\to\infty} \int -\frac{1}{n} \log \sum_{m=0}^{n^2}
\overline{p}^{(m)}(e,\overline{X}_n)\, d\Prob\\
&=& \lim_{n\to\infty} -\frac{1}{n} \sum_{w\in W} \overline{p}^{(n)}(e,w) \log
 \sum_{m=0}^{n^2}
\overline{p}^{(m)}(e,w).
\end{eqnarray*}
Non-negativity of the Kullback-Leibler divergence (also called
\textit{Shannon's inequality} in this context) gives
$$
-\sum_{w\in W} \overline{p}^{(n)}(e,w) \log \mu(w) \geq -\sum_{w\in W} \overline{p}^{(n)}(e,w) \log \overline{p}^{(n)}(e,w)
$$
for every finitely supported probability measure $\mu$ on $W$. 
We apply now this inequality to the probability measure $\mu_0$ defined by
$$
\mu_0(w):=\frac{1}{n^2+1} \sum_{m=0}^{n^2} \overline{p}^{(m)}(e,w)\quad
\textrm{ for $w\in W$}: 
$$
\begin{eqnarray*}
h_G& \geq & \limsup_{n\to\infty}\left( -\frac{1}{n} \sum_{w\in W}
\overline{p}^{(n)}(e,w) \log (n^2+1) -\frac{1}{n} \sum_{w\in W} \overline{p}^{(n)}(e,w) \log
\overline{p}^{(n)}(e,w)\right) \\
&=& \limsup_{n\to\infty} -\frac{1}{n}\int \log \overline{\pi}_n(\overline{X}_n)\, d\Prob.
\end{eqnarray*}
Now we can conclude with Fatou's Lemma and (\ref{equ:liminf-h}):
\begin{eqnarray}
h_G & \leq & \int \liminf_{n\to\infty} -\frac1n \log \overline{\pi}_n(\overline{X}_n) d\Prob \leq 
\liminf_{n\to\infty} \int -\frac1n \log \overline{\pi}_n(\overline{X}_n) d\Prob \nonumber\\
& \leq &
\limsup_{n\to\infty} \int -\frac1n \log \overline{\pi}_n(\overline{X}_n) d\Prob \leq h_G.\label{eq:entropy-final}
\end{eqnarray}
Thus, $\overline{h}=\lim_{n\to\infty} -\frac{1}{n} \E\bigl[\log \overline{\pi}_n(\overline{X}_n)\bigr]$ exists
and the limit equals $h_G$. It follows immediately from Theorem \ref{thm:l-entropy} that $\overline{h}$ is also the rate of escape with respect to the Green distance.
\end{proof}
\begin{remark}
The results of Forghani \cite{Forghani} are somehow in the same spirit as Theorem \ref{thm:L-limit}. It is shown there that  the asymptotic entropy of transformed random walks (by stopping times) equals the product of the original entropy and the expectation of the corresponding stopping times.
\end{remark}
Furthermore: 
\begin{cor}\label{cor:overline-h-convergence-types}
We have the following types of convergence: 
\begin{enumerate}
\item 
$$
\overline{h}=\liminf_{n\to\infty} -\frac1n \log \overline{\pi}_n(\overline{X}_n)\quad
\textrm{ almost surely.}
$$
\item 
$$
-\frac{1}{n}\log \overline{\pi}_n(\overline{X}_n) \xrightarrow{L_1} \overline{h}.
$$
\end{enumerate}
\end{cor}
\begin{proof}
The proofs are analogous to the proofs in \cite[Corollary 3.9, Lemma
3.10]{gilch:11}.
\end{proof}


\section{Asymptotic entropy of the random walk on the building}

\label{sec:entropy-building}

In this section we deduce existence and a formula for the asymptotic entropy of
the random walk on the building from the entropy of the retracted walk. The
following lemma establishes an important link between transition probabilities
of both random walks:

\begin{Lemma}\label{lemma:p-to-overline-p}
For all $x\in\Delta$, $n\geq 1$:
$$
\overline{p}^{(n)}\bigl(e,\rho(x)\bigr) = p^{(n)}(o,x)   q_{\rho(x)}
$$
\end{Lemma}
\begin{proof}
Recall that $p^{(n)}(o,x)=p^{(n)}_w/q_w$, where $w=\delta(o,x)$.
According to \cite[Proposition 4.5]{gilch-mueller-parkinson:14} (applied to $P^n$)  we have
$$
\overline{p}^{(n)}\bigl(e,\rho(x)\bigr) = \sum_{w\in W} a^e_{\rho(x),w}
q_w^{-1} p_w,
$$
where $a^e_{\rho(x),w}:= | \Delta_{\rho(x)}(o)\cap \Delta_w(o)|$ with
$\Delta_u(o)=\{y\in\Delta \mid \delta(o,y)=u\}$ for $u\in W$. If $y\in
\Delta_{\rho(x)}(o)\cap \Delta_w(o)$ then we must have $\rho(x)=\delta(o,y)=w$, and
therefore $a^e_{\rho(x),w}=|\Delta_{\rho(x)}(o)|=q_{\rho(x)}$. That
is,
$$
\overline{p}^{(n)}\bigl(e,\rho(x)\bigr) = q_{\rho(x)} q_{\rho(x)}^{-1}
p^{(n)}_{\rho(x)}=p^{(n)}(o,x) q_{\rho(x)}.
$$
\end{proof}
The following result gives the additional asymptotic information when switching
from the retracted walk to the random walk on the building:
\begin{Prop}\label{prop:hq-convergence}
$$
 \frac1n
\log q_{\overline{X}_n}
\quad \xrightarrow{n\to\infty} \quad \frac{\E_{\nu}[\log(q_{\cR_{1}^{-1} \cR_{2}})]}{\E_{\nu}[\e{2}-\e{1}]}:=h_{q}\quad \textrm{almost surely}.
$$
\end{Prop}
\begin{proof}
Recall the definition of $\mathbf{k}(n)$ from (\ref{equ:def-k(n)}) and recall that $\mathcal{R}_k$ is the root of the cone associated with
$\overline{X}_{\e{k}}$. In particular, there is a shortest path from
$o$ to  $\overline{X}_{\e{k}}$ passing through $\mathcal{R}_k$. This yields
$$
\log q_{\mathcal{R}_{\mathbf{k}(n)}}\leq 
\log q_{\overline{X}_{n}} = \log q_{\mathcal{R}_{\mathbf{k}(n)}}+\log
q_{\mathcal{R}_{\mathbf{k}(n)}^{-1}\overline{X}_{n}}.
$$
Moreover
$$
\log q_{\mathcal{R}_k} = \sum_{i=0}^k \log q_{\mathcal{R}_{i-1}^{-1}\mathcal{R}_i},
$$
where $\mathcal{R}_{-1}:=e$. Then by the ergodic theorem
$$
\log q_{\mathcal{R}_k} \quad \xrightarrow{k\to\infty} \quad \E_{\nu}[\log(q_{\cR_{1}^{-1} \cR_{2}})]\quad \textrm{ almost surely.}
$$
Furthermore, we have
$$
1\leq q_{\mathcal{R}_{\mathbf{k}(n)}^{-1}\overline{X}_{n}}\leq \bigl(\max_{s\in S} q_s\bigr)^{d(\mathcal{R}_{\mathbf{k}(n)},\overline{X}_{n})}.
$$
By Lemma \ref{lem:root-last-entry-distance} and due to $\mathbf{k}(n)\leq n$, we have
\begin{equation}\label{equ:dist1}
\frac{d(\mathcal{R}_{\mathbf{k}(n)},\overline{X}_{\e{\mathbf{k}(n)}})}{n}
=
\frac{d(\mathcal{R}_{\mathbf{k}(n)},\overline{X}_{\e{\mathbf{k}(n)}})}{\mathbf{k}(n)}\frac{\mathbf{k}(n)}{n}
\xrightarrow{n\to\infty} 0 \quad \textrm{almost surely.}
\end{equation}
Since $0\leq d(\overline{X}_{\e{\mathbf{k}(n)}},\overline{X}_n)\leq
(n-\e{\mathbf{k}(n)})L_0$, we have due to (\ref{equ:n-e-convergence}) that
$\frac1n d(\overline{X}_{\e{\mathbf{k}(n)}},\overline{X}_n)\to 0$ as
$n\to\infty$ almost surely.
This in turn implies together with (\ref{equ:dist1}) that
$$
\frac{d(\mathcal{R}_{\mathbf{k}(n)},\overline{X}_n)}{n}\xrightarrow{n\to\infty} 0 \quad \textrm{almost surely.}
$$
Therefore,
$$
\frac{1}{n} \log q_{\mathcal{R}_{\mathbf{k}(n)}^{-1}\overline{X}_{n}}
\xrightarrow{n\to\infty} 0 \quad 
\textrm{ almost surely}.
$$
Observe that from the proof of Theorem \ref{thm:l-entropy} follows that
$\mathbf{k}(n)/n\to \E_{\nu}[\e{2}-\e{1}]^{-1}$ as $n\to\infty$. Now we can conclude:
$$
\lim_{n\to\infty} \frac1n \log q_{\overline{X}_n}=\lim_{n\to\infty}
\frac{\mathbf{k}(n)}{n}\frac{1}{\mathbf{k}(n)} \log
q_{\mathcal{R}_{\mathbf{k}(n)}}=  \frac{\E_{\nu}[\log(q_{\cR_{1}^{-1} \cR_{2}})]}{\E_{\nu}[\e{2}-\e{1}]}.
$$
\end{proof}
By (\ref{equ:drift-formula}) we can also write
$$
h_q=\frac{\E_{\nu}[\log(q_{\cR_{1}^{-1} \cR_{2}})] \cdot v}{\mathbb{E}_\nu[d(\mathcal{R}_2,\mathcal{R}_1)]}.
$$
Since $0\leq \frac1n \log q_{\overline{X}_n}\leq \max_{s\in S} q_s$ we also
have $\lim_{n\to\infty} \frac1n \mathbb{E}\bigr[\log
q_{\overline{X}_n}\bigr]=h_q$. Now we can conclude:
\begin{thm}\label{thm:building-entropy}
The asymptotic entropy of the random walk on the building exists and is given by
$$
h=\overline{h}+h_q.
$$
\end{thm}
\begin{proof}
Recall that Lemma \ref{lemma:p-to-overline-p} gives 
$\overline{p}^{(n)}\bigl(e,\overline{X}_n\bigr) = p^{(n)}(o,X_n)  q_{\overline{X}_n}$,
implying
$$
-\frac1n \mathbb{E}\bigl[\log
\overline{p}^{(n)}\bigl(e,\overline{X}_n\bigr)\bigr] = -\frac1n \mathbb{E}\bigl[\log
p^{(n)}\bigl(o,\overline{X}_n\bigr)\bigr]  -\frac1n
\mathbb{E}\bigl[\log q_{\overline{X}_n}\bigr].
$$
Together with Theorem \ref{thm:L-limit} and Proposition \ref{prop:hq-convergence} we get the proposed claim.
\end{proof}

We note that an analogous result to Theorem~\ref{thm:building-entropy} is obtained in Ledrappier and Lim \cite[Theorem~1.1]{LL:10} for \textit{volume entropy} of a hyperbolic building. These concepts are quite different, although it is interesting to note the similar forms of the formulae.

\begin{cor}\label{cor:buildinggreen}
The entropy equals the rate of escape with respect to the Green distance of the random
walk $(X_n)_{n\geq 0}$.
\end{cor}
\begin{proof}
By Lemma \ref{lemma:p-to-overline-p} we get
$$
\log \sum_{m\geq 0} p^{(m)}(o,X_n)=\log \sum_{m\geq 0}
\frac{\overline{p}^{(m)}(e,\overline{X}_n)}{q_{\overline{X}_n}}= \log
G(e,\overline{X}_n|1) -\log q_{\overline{X}_n}
$$
and the claim follows now with Theorems \ref{thm:L-limit} and \ref{thm:building-entropy}.
\end{proof}
Immediate consequences of Corollary \ref{cor:overline-h-convergence-types} are the following.
\begin{cor}\label{cor:h-convergence-types}
We have the following types of convergence:
\begin{enumerate}
\item 
$$
h=\liminf_{n\to\infty} -\frac1n \log \pi_n(X_n)\quad
\textrm{ almost surely.}
$$
\item 
$$
-\frac{1}{n}\log \pi_n(X_n) \xrightarrow{L_1} h.
$$
\end{enumerate}
\end{cor}
\begin{appendix}

\section{Normal form automata}\label{app:A}

The aim of this appendix is to prove Theorem~\ref{thm:strongly-connected}. That is, we show that for each Fuchsian Coxeter system $(W,S)$ the strongly automatic structure $\mathfrak{A}(W,S)$ is strongly connected. In this appendix we refer to $\mathfrak{A}=\mathfrak{A}(W,S)$ as the \textit{normal form} automata for $(W,S)$. Our starting point is \cite[Appendix~A]{gilch-mueller-parkinson:14}, where the Cannon automaton $\cA(W,S)$ for each Fuchsian Coxeter system is explicitly constructed. 

\begin{proof}[Proof of Theorem~\ref{thm:strongly-connected}]
It is convenient to divide the set of all Fuchsian Coxeter systems $(W,S)$ into $4$ classes. First consider triangle groups (that is, $|S|=3$). Thus $S=\{s,t,u\}$, and we write $a=m_{st}$, $b=m_{tu}$, and $c=m_{us}$. Renaming the generators if necessary we may assume that $a\geq b\geq c\geq 2$. Then $(W,S)$ is Fuchsian if and only if one of the following occurs:
\begin{itemize}
\item $a\geq b\geq c\geq 3$ with $a\neq 3$ (we call these \textit{Class I} groups). 
\item $a\geq b\geq 4$ and $c=2$ with $a>4$ (we call these \textit{Class II} groups).
\item $a> 6$, $b=3$, and $c=2$ (we call these \textit{Class III} groups). 
\end{itemize}
The remaining Fuchsian Coxeter systems are those with $|S|\geq 4$. We call these Fuchsian Coxeter systems of \textit{Class IV}. 

We now construct the normal form of the Cannon automaton for each class (the diagrams in \cite[Appendix~A]{gilch-mueller-parkinson:14} are useful here). Suppose that $(W,S)$ is a Class I Fuchsian Coxeter system. It follows from Tits' solution to the word problem in Coxeter groups~\cite{Ti:69} 
that a normal form automaton for $(W,S)$ is obtained by removing the following three arrows from the Cannon automaton
\begin{align*}
\underbrace{tst\cdots}_{\text{$a-1$ terms}}&\rightarrow\,\, \underbrace{tst\cdots}_{\text{$a$ terms}} & \underbrace{utu\cdots}_{\text{$b-1$ terms}}&\rightarrow\,\, \underbrace{utu\cdots}_{\text{$b$ terms}} & \underbrace{usu\cdots}_{\text{$c-1$ terms}}&\rightarrow\,\, \underbrace{usu\cdots}_{\text{$c$ terms}}
\end{align*}
This is illustrated below in Figure~\ref{fig:caseI}.

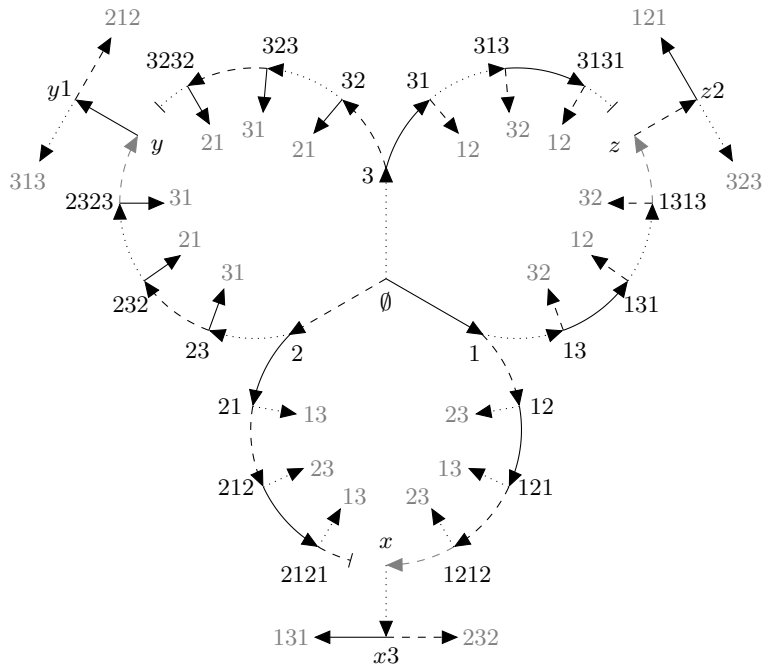
\begin{figure}[!h]
\begin{minipage}{0.65\textwidth}
\centering
\begin{tikzpicture} [scale=6]
\draw [->,dotted,-triangle 45] (0,0) -- (0,0.25); 
\draw [->,-triangle 45] (0,0) -- ({0.25*cos(-30)},{0.25*sin(-30)});
\draw [->,dashed,-triangle 45] (0,0) -- ({0.25*cos(210)},{0.25*sin(210)});
\draw [->,domain=135:170,-triangle 45] plot ({0.3*cos(\x)},{0.3*sin(\x)-0.335});
\draw [->,domain=170:205,dashed,-triangle 45] plot ({0.3*cos(\x)},{0.3*sin(\x)-0.335});
\draw [->,domain=205:240,-triangle 45] plot ({0.3*cos(\x)},{0.3*sin(\x)-0.335});
\draw [-|,dashed,domain=240:255] plot ({0.3*cos(\x)},{0.3*sin(\x)-0.335});
\draw [->,domain=45:10,dashed,-triangle 45] plot ({0.3*cos(\x)},{0.3*sin(\x)-0.335});
\draw [->,domain=10:-25,-triangle 45] plot ({0.3*cos(\x)},{0.3*sin(\x)-0.335});
\draw [->,domain=-25:-60,dashed,-triangle 45] plot ({0.3*cos(\x)},{0.3*sin(\x)-0.335});
\draw [->,gray,dashed,domain=-60:-90,-triangle 45] plot ({0.3*cos(\x)},{0.3*sin(\x)-0.335});
\draw [->,dotted,-triangle 45] ({0.3*cos(-90)},{0.3*sin(-90)-0.335}) -- ({0.3*cos(-90)},{0.3*sin(-90)-0.335-0.16});
\draw [->,-triangle 45] ({0.3*cos(-90)},{0.3*sin(-90)-0.335-0.16}) -- ({0.3*cos(-90)-0.16},{0.3*sin(-90)-0.335-0.16});
\draw [->,dashed,-triangle 45] ({0.3*cos(-90)},{0.3*sin(-90)-0.335-0.16}) -- ({0.3*cos(-90)+0.16},{0.3*sin(-90)-0.335-0.16});
\draw [->,dotted,-triangle 45] ({0.3*cos(170)},{0.3*sin(170)-0.335}) -- ({0.2*cos(170)},{0.2*sin(170)-0.335});
\draw [->,dotted,-triangle 45] ({0.3*cos(205)},{0.3*sin(205)-0.335}) -- ({0.2*cos(205)},{0.2*sin(205)-0.335});
\draw [->,dotted,-triangle 45] ({0.3*cos(240)},{0.3*sin(240)-0.335}) -- ({0.2*cos(240)},{0.2*sin(240)-0.335});
\draw [->,dotted,-triangle 45] ({0.3*cos(10)},{0.3*sin(10)-0.335}) -- ({0.2*cos(10)},{0.2*sin(10)-0.335});
\draw [->,dotted,-triangle 45] ({0.3*cos(-25)},{0.3*sin(-25)-0.335}) -- ({0.2*cos(-25)},{0.2*sin(-25)-0.335});
\draw [->,dotted,-triangle 45] ({0.3*cos(-60)},{0.3*sin(-60)-0.335}) -- ({0.2*cos(-60)},{0.2*sin(-60)-0.335});
\path (0,-0.05) node {\footnotesize{$\emptyset$}}
 ({0.25*cos(-30)-0.02},{0.25*sin(-30)-0.04}) node {\footnotesize{$1$}}
  ({-0.25*cos(-30)+0.02},{0.25*sin(-30)-0.04}) node {\footnotesize{$2$}}
  (-0.04,0.23) node {\footnotesize{$3$}}
   ({0.3*cos(10)+0.05},{0.3*sin(10)-0.335}) node {\footnotesize{$12$}}
    ({0.3*cos(-25)+0.06},{0.3*sin(-25)-0.335}) node {\footnotesize{$121$}}
   ({0.3*cos(-60)+0.03},{0.3*sin(-60)-0.335-0.06}) node {\footnotesize{$1212$}}
    ({-0.3*cos(10)-0.05},{0.3*sin(10)-0.335}) node {\footnotesize{$21$}}
    ({-0.3*cos(-25)-0.06},{0.3*sin(-25)-0.335}) node {\footnotesize{$212$}}
    ({-0.3*cos(-60)-0.03},{0.3*sin(-60)-0.335-0.06}) node {\footnotesize{$2121$}}
    ({-0.3*cos(-90)},{0.3*sin(-90)-0.335+0.05}) node {\footnotesize{$x$}}
      ({-0.3*cos(-90)},{0.3*sin(-90)-0.335+0.05-0.25}) node {\footnotesize{$x3$}}
        ({-0.3*cos(-90)+0.21},{0.3*sin(-90)-0.335-0.16}) node [gray]{\footnotesize{$232$}}
        ({-0.3*cos(-90)-0.21},{0.3*sin(-90)-0.335-0.16}) node [gray]{\footnotesize{$131$}}
       ({0.2*cos(10)-0.04},{0.2*sin(10)-0.335}) node [gray]{\footnotesize{$23$}}
    ({0.2*cos(-25)-0.04},{0.2*sin(-25)-0.335}) node [gray]{\footnotesize{$13$}}
    ({0.2*cos(-60)-0.04+0.01},{0.2*sin(-60)-0.335+0.025}) node [gray]{\footnotesize{$23$}}
    ({-0.2*cos(10)+0.04},{0.2*sin(10)-0.335}) node [gray]{\footnotesize{$13$}}
    ({-0.2*cos(-25)+0.04},{0.2*sin(-25)-0.335}) node [gray]{\footnotesize{$23$}}
    ({-0.2*cos(-60)+0.04-0.01},{0.2*sin(-60)-0.335+0.025}) node [gray] {\footnotesize{$13$}};
\begin{scope}[rotate=120]
\draw [->,domain=135:170,dotted,-triangle 45] plot ({0.3*cos(\x)},{0.3*sin(\x)-0.335});
\draw [->,domain=170:205,-triangle 45] plot ({0.3*cos(\x)},{0.3*sin(\x)-0.335});
\draw [->,domain=205:240,dotted,-triangle 45] plot ({0.3*cos(\x)},{0.3*sin(\x)-0.335});
\draw [->,gray,dashed,domain=240:270,-triangle 45] plot ({0.3*cos(\x)},{0.3*sin(\x)-0.335});
\draw [->,domain=45:10,-triangle 45] plot ({0.3*cos(\x)},{0.3*sin(\x)-0.335});
\draw [->,domain=10:-25,dotted,-triangle 45] plot ({0.3*cos(\x)},{0.3*sin(\x)-0.335});
\draw [->,domain=-25:-60,-triangle 45] plot ({0.3*cos(\x)},{0.3*sin(\x)-0.335});
\draw [-|, dotted,domain=-60:-75] plot ({0.3*cos(\x)},{0.3*sin(\x)-0.335});
\draw [->,dashed,-triangle 45] ({0.3*cos(-90)},{0.3*sin(-90)-0.335}) -- ({0.3*cos(-90)},{0.3*sin(-90)-0.335-0.16});
\draw [->,dotted,-triangle 45] ({0.3*cos(-90)},{0.3*sin(-90)-0.335-0.16}) -- ({0.3*cos(-90)-0.16},{0.3*sin(-90)-0.335-0.16});
\draw [->,-triangle 45] ({0.3*cos(-90)},{0.3*sin(-90)-0.335-0.16}) -- ({0.3*cos(-90)+0.16},{0.3*sin(-90)-0.335-0.16});
\draw [->,dashed,-triangle 45] ({0.3*cos(170)},{0.3*sin(170)-0.335}) -- ({0.2*cos(170)},{0.2*sin(170)-0.335});
\draw [->,dashed,-triangle 45] ({0.3*cos(205)},{0.3*sin(205)-0.335}) -- ({0.2*cos(205)},{0.2*sin(205)-0.335});
\draw [->,dashed,-triangle 45] ({0.3*cos(240)},{0.3*sin(240)-0.335}) -- ({0.2*cos(240)},{0.2*sin(240)-0.335});
\draw [->,dashed,-triangle 45] ({0.3*cos(10)},{0.3*sin(10)-0.335}) -- ({0.2*cos(10)},{0.2*sin(10)-0.335});
\draw [->,dashed,-triangle 45] ({0.3*cos(-25)},{0.3*sin(-25)-0.335}) -- ({0.2*cos(-25)},{0.2*sin(-25)-0.335});
\draw [->,dashed,-triangle 45] ({0.3*cos(-60)},{0.3*sin(-60)-0.335}) -- ({0.2*cos(-60)},{0.2*sin(-60)-0.335});
\path 
   ({0.3*cos(10)+0.05},{0.3*sin(10)-0.335}) node {\footnotesize{$31$}}
    ({0.3*cos(-25)+0.06},{0.3*sin(-25)-0.335}) node {\footnotesize{$313$}}
    ({0.3*cos(-60)+0.03},{0.3*sin(-60)-0.335-0.06}) node {\footnotesize{$3131$}}
    ({-0.3*cos(10)-0.05},{0.3*sin(10)-0.335}) node {\footnotesize{$13$}}
    ({-0.3*cos(-25)-0.06},{0.3*sin(-25)-0.335}) node {\footnotesize{$131$}}
    ({-0.3*cos(-60)-0.03},{0.3*sin(-60)-0.335-0.06}) node {\footnotesize{$1313$}}
    ({-0.3*cos(-90)},{0.3*sin(-90)-0.335+0.05}) node {\footnotesize{$z$}}
      ({-0.3*cos(-90)},{0.3*sin(-90)-0.335+0.05-0.25}) node {\footnotesize{$z2$}}
        ({-0.3*cos(-90)+0.21},{0.3*sin(-90)-0.335-0.16}) node [gray] {\footnotesize{$121$}}
        ({-0.3*cos(-90)-0.21},{0.3*sin(-90)-0.335-0.16}) node [gray] {\footnotesize{$323$}}
       ({0.2*cos(10)-0.04},{0.2*sin(10)-0.335}) node [gray] {\footnotesize{$12$}}
    ({0.2*cos(-25)-0.04},{0.2*sin(-25)-0.335}) node [gray] {\footnotesize{$32$}}
    ({0.2*cos(-60)-0.04+0.01},{0.2*sin(-60)-0.335+0.025})  node [gray] {\footnotesize{$12$}}
    ({-0.2*cos(10)+0.04},{0.2*sin(10)-0.335}) node [gray]{\footnotesize{$32$}}
    ({-0.2*cos(-25)+0.04},{0.2*sin(-25)-0.335}) node [gray] {\footnotesize{$12$}}
    ({-0.2*cos(-60)+0.04-0.01},{0.2*sin(-60)-0.335+0.025})  node [gray] {\footnotesize{$32$}};
\end{scope}
\begin{scope}[rotate=-120]
\draw [->,domain=135:170,dashed,-triangle 45] plot ({0.3*cos(\x)},{0.3*sin(\x)-0.335});
\draw [->,domain=170:205,dotted,-triangle 45] plot ({0.3*cos(\x)},{0.3*sin(\x)-0.335});
\draw [->,domain=205:240,dashed,-triangle 45] plot ({0.3*cos(\x)},{0.3*sin(\x)-0.335});
\draw [-|,dotted,domain=240:255] plot ({0.3*cos(\x)},{0.3*sin(\x)-0.335});
\draw [->,domain=45:10,dotted,-triangle 45] plot ({0.3*cos(\x)},{0.3*sin(\x)-0.335});
\draw [->,domain=10:-25,dashed,-triangle 45] plot ({0.3*cos(\x)},{0.3*sin(\x)-0.335});
\draw [->,domain=-25:-60,dotted,-triangle 45] plot ({0.3*cos(\x)},{0.3*sin(\x)-0.335});
\draw [->,gray,dashed,domain=-60:-90,-triangle 45] plot ({0.3*cos(\x)},{0.3*sin(\x)-0.335});
\draw [->,-triangle 45] ({0.3*cos(-90)},{0.3*sin(-90)-0.335}) -- ({0.3*cos(-90)},{0.3*sin(-90)-0.335-0.16});
\draw [->,dashed,-triangle 45] ({0.3*cos(-90)},{0.3*sin(-90)-0.335-0.16}) -- ({0.3*cos(-90)-0.16},{0.3*sin(-90)-0.335-0.16});
\draw [->,dotted,-triangle 45] ({0.3*cos(-90)},{0.3*sin(-90)-0.335-0.16}) -- ({0.3*cos(-90)+0.16},{0.3*sin(-90)-0.335-0.16});
\draw [->,-triangle 45] ({0.3*cos(170)},{0.3*sin(170)-0.335}) -- ({0.2*cos(170)},{0.2*sin(170)-0.335});
\draw [->,-triangle 45] ({0.3*cos(205)},{0.3*sin(205)-0.335}) -- ({0.2*cos(205)},{0.2*sin(205)-0.335});
\draw [->,-triangle 45] ({0.3*cos(240)},{0.3*sin(240)-0.335}) -- ({0.2*cos(240)},{0.2*sin(240)-0.335});
\draw [->,-triangle 45] ({0.3*cos(10)},{0.3*sin(10)-0.335}) -- ({0.2*cos(10)},{0.2*sin(10)-0.335});
\draw [->,-triangle 45] ({0.3*cos(-25)},{0.3*sin(-25)-0.335}) -- ({0.2*cos(-25)},{0.2*sin(-25)-0.335});
\draw [->,-triangle 45] ({0.3*cos(-60)},{0.3*sin(-60)-0.335}) -- ({0.2*cos(-60)},{0.2*sin(-60)-0.335});
\path 
   ({0.3*cos(10)+0.05},{0.3*sin(10)-0.335}) node {\footnotesize{$23$}}
    ({0.3*cos(-25)+0.06},{0.3*sin(-25)-0.335}) node {\footnotesize{$232$}}
    ({0.3*cos(-60)+0.03},{0.3*sin(-60)-0.335-0.06}) node {\footnotesize{$2323$}}
    ({-0.3*cos(10)-0.05},{0.3*sin(10)-0.335}) node {\footnotesize{$32$}}
    ({-0.3*cos(-25)-0.06},{0.3*sin(-25)-0.335}) node {\footnotesize{$323$}}
   ({-0.3*cos(-60)-0.03},{0.3*sin(-60)-0.335-0.06}) node {\footnotesize{$3232$}}
    ({-0.3*cos(-90)},{0.3*sin(-90)-0.335+0.05}) node {\footnotesize{$y$}}
      ({-0.3*cos(-90)},{0.3*sin(-90)-0.335+0.05-0.25}) node {\footnotesize{$y1$}}
        ({-0.3*cos(-90)+0.21},{0.3*sin(-90)-0.335-0.16}) node [gray] {\footnotesize{$313$}}
        ({-0.3*cos(-90)-0.21},{0.3*sin(-90)-0.335-0.16}) node [gray] {\footnotesize{$212$}}
       ({0.2*cos(10)-0.04},{0.2*sin(10)-0.335}) node  [gray] {\footnotesize{$31$}}
    ({0.2*cos(-25)-0.04},{0.2*sin(-25)-0.335}) node [gray] {\footnotesize{$21$}}
    ({0.2*cos(-60)-0.04+0.01},{0.2*sin(-60)-0.335+0.025})  node [gray] {\footnotesize{$31$}}
    ({-0.2*cos(10)+0.04},{0.2*sin(10)-0.335}) node [gray] {\footnotesize{$21$}}
    ({-0.2*cos(-25)+0.04},{0.2*sin(-25)-0.335}) node [gray] {\footnotesize{$31$}}
    ({-0.2*cos(-60)+0.04-0.01},{0.2*sin(-60)-0.335+0.025})  node [gray] {\footnotesize{$21$}};
\end{scope}
\end{tikzpicture}
\caption{Normal form automaton for Class I}\label{fig:caseI}
\end{minipage}\begin{minipage}{0.3\textwidth}
The generators are labelled $1$, $2$, and~$3$, and the labels on the edges are indicated by line styles, with $1$ being solid, $2$ being dashed, and $3$ being dotted. The cone types are given by the base element of a representative cone of that type. The cone types in grey are duplicates, and the reader should instead imagine the arrow pointing to the corresponding cone type in black. Furthermore, $x=121\cdots$, $y=232\cdots$ and $z=131\cdots$ are the longest elements of the parabolic subgroups $W_{12}$, $W_{23}$, and $W_{13}$ respectively, and the `deleted' arrows (from the Cannon automaton) are indicated with $\dashv$. 
\end{minipage}
\end{figure}

Similarly, the normal form automaton for a Class II Fuchsian Coxeter system is obtained by removing the following arrows from the Cannon automaton:
\begin{align*}
xs&\xrightarrow{s} x & xsu&\xrightarrow{s}xu&yu&\xrightarrow{u}y & yus&\xrightarrow{u}ys&u&\xrightarrow{s}su
\end{align*}
(where $x=sts\cdots$ is the longest element of $\langle s,t\rangle$ and $y=tut\cdots$ is the longest element of $\langle t,u\rangle$) and for Class III we remove the following arrows:
\begin{align*}
xs&\xrightarrow{s}x&xsu&\xrightarrow{s}xu&xtut&\xrightarrow{u}xut&xtuts&\xrightarrow{u}xuts&xtus&\xrightarrow{s}xtu&ut&\xrightarrow{u}tut&uts&\xrightarrow{u}tuts&u&\xrightarrow{s}su
\end{align*}
(where $x=sts\cdots$ is the longest element of $\langle s,t\rangle$). The resulting automata are illustrated in the figures below (for $b=4$ in Class II there are some very minor modifications required, and the reader is referred to \cite[Appendix~A]{gilch-mueller-parkinson:14}).

\begin{figure}[!h]
\centering
\begin{tikzpicture} [scale=2.7]
\node {\begin{tikzpicture} [scale=2.5]
\draw [dashed,->,-triangle 45] (0,-0.21) -- ({cos(23.15)-0.92},{sin(23.15)});
\draw [dotted,->,-triangle 45] (0,-0.21) -- ({cos(310)-0.92},{sin(310)});
\draw [->,-triangle 45] (0,-0.21) -- ({-cos(310)+0.92},{sin(310)});
%
\draw [dotted,->,-triangle 45] ({-cos(310)+0.92},{sin(310)}) -- (0,-1.15);
\draw [dashed,->,-triangle 45] (0,-1.15) -- (0,-1.5);
\draw [->,-triangle 45] (0,-1.5) -- (-0.35,-1.5);
\draw [dotted,->,-triangle 45] (0,-1.5) -- (0.35,-1.5);
\draw [dotted,->,domain=23.15:45,-triangle 45] plot ({cos(\x)-0.92},{sin(\x)});
\draw [dashed,->,domain={45}:{45+20},-triangle 45] plot ({cos(\x)-0.92},{sin(\x)});
\draw [dotted,->,domain={45+1*20}:{45+20+1*20},-triangle 45] plot ({cos(\x)-0.92},{sin(\x)});
\draw [dashed,->,domain={45+2*20}:{45+20+2*20},-triangle 45] plot ({cos(\x)-0.92},{sin(\x)});
\draw [dashed,->,domain=160:180,-triangle 45] plot ({cos(\x)-0.92},{sin(\x)});
\draw [dotted,->,domain=140:160,-triangle 45] plot ({cos(\x)-0.92},{sin(\x)});
\draw [->,gray,dashed,domain=105:140,-triangle 45] plot ({cos(\x)-0.92},{sin(\x)});
\draw [->,domain=156.85:135,-triangle 45] plot ({cos(\x)+0.92},{sin(\x)});
\draw [dashed,->,domain={135}:{135-20},-triangle 45] plot ({cos(\x)+0.92},{sin(\x)});
\draw [->,domain={135-1*20}:{135-20-1*20},-triangle 45] plot ({cos(\x)+0.92},{sin(\x)});
\draw [dashed,->,domain={135-2*20}:{135-20-2*20},-triangle 45] plot ({cos(\x)+0.92},{sin(\x)});
\draw [dashed,->,domain=40:20,-triangle 45] plot ({cos(\x)+0.92},{sin(\x)});
\draw [->,gray,dashed,domain=75:40,-triangle 45] plot ({cos(\x)+0.92},{sin(\x)});
\draw [dashed,->,domain={310}:{290},-triangle 45] plot ({cos(\x)-0.92},{sin(\x)});
\draw [dotted,->,domain={310-1*20}:{290-1*20},-triangle 45] plot ({cos(\x)-0.92},{sin(\x)});
\draw [dashed,->,domain={310-2*20}:{290-2*20},-triangle 45] plot ({cos(\x)-0.92},{sin(\x)});
\draw [dashed,->,domain={220}:{200},-triangle 45] plot ({cos(\x)-0.92},{sin(\x)});
\draw [->,gray,dashed,domain=250:220,-triangle 45] plot ({cos(\x)-0.92},{sin(\x)});
\draw [dashed,->,domain={230}:{250},-triangle 45] plot ({cos(\x)+0.92},{sin(\x)});
\draw [->,domain={230+1*20}:{250+1*20},-triangle 45] plot ({cos(\x)+0.92},{sin(\x)});
\draw [dashed,->,domain={230+2*20}:{250+2*20},-triangle 45] plot ({cos(\x)+0.92},{sin(\x)});
\draw [dashed,->,domain={-20}:{0},-triangle 45] plot ({cos(\x)+0.92},{sin(\x)});
\draw [->,domain={-40}:{-20},-triangle 45] plot ({cos(\x)+0.92},{sin(\x)});
\draw [->,gray,dashed,domain={-70}:{-40},-triangle 45] plot ({cos(\x)+0.92},{sin(\x)});
\draw [->,-triangle 45] (-1.92,0) -- ({-1.92-0.35},0);
\draw [dashed,->,-triangle 45] ({-1.92-0.35},0) --  ({-1.92-2*0.35},0);
\draw [dotted,->,-triangle 45] ({-1.92-2*0.35},0) --  ({-1.92-3*0.35},0);
\draw [->,-triangle 45] ({-1.92-2*0.35},0) -- ({-1.92-2*0.35},0.35);
\draw [->,-triangle 45] ({cos(200)-0.92},{sin(200)}) -- ({-1-0.92-0.35},{sin(200)});
%
\draw [dashed,->,-triangle 45] ({-1-0.92-0.35},{sin(200)}) -- ({-1-0.92-2*0.35},{sin(200)});
\draw [dotted,->,-triangle 45] (1.92,0) -- ({1.92+0.35},0);
\draw [dashed,->,-triangle 45] ({+1.92+0.35},0) --  ({+1.92+2*0.35},0);
\draw [->,-triangle 45] ({+1.92+2*0.35},0) --  ({+1.92+3*0.35},0);
\draw [dotted,->,-triangle 45] ({+1.92+2*0.35},0) -- ({+1.92+2*0.35},-0.35);
\draw [dotted,->,-triangle 45] ({-cos(160)+0.92},{sin(160)}) -- ({1+0.92+0.35},{sin(160)});
%
\draw [dashed,->,-triangle 45] ({+1+0.92+0.35},{sin(160)}) -- ({+1+0.92+2*0.35},{sin(160)});
\draw [->,-triangle 45] ({cos(45)-0.92},{sin(45)}) -- ({0.7*cos(45)-0.92},{0.7*sin(45)});
\draw [->,-triangle 45] ({cos(65)-0.92},{sin(65)}) -- ({0.7*cos(65)-0.92},{0.7*sin(65)});
\draw [->,-triangle 45] ({cos(85)-0.92},{sin(85)}) -- ({0.7*cos(85)-0.92},{0.7*sin(85)});
\draw [->,-triangle 45] ({cos(105)-0.92},{sin(105)}) -- ({0.7*cos(105)-0.92},{0.7*sin(105)});
\draw [->,-triangle 45] ({cos(140)-0.92},{sin(140)}) -- ({0.7*cos(140)-0.92},{0.7*sin(140)});
\draw [->,-triangle 45] ({cos(160)-0.92},{sin(160)}) -- ({0.7*cos(160)-0.92},{0.7*sin(160)});
\draw [->,-triangle 45] ({cos(220)-0.92},{sin(220)}) -- ({0.7*cos(220)-0.92},{0.7*sin(220)});
\draw [->,-triangle 45] ({cos(250)-0.92},{sin(250)}) -- ({0.7*cos(250)-0.92},{0.7*sin(250)});
\draw [->,-triangle 45] ({cos(270)-0.92},{sin(270)}) -- ({0.7*cos(270)-0.92},{0.7*sin(270)});
\draw [->,-triangle 45] ({cos(290)-0.92},{sin(290)}) -- ({0.7*cos(290)-0.92},{0.7*sin(290)});
\draw [dotted,->,-triangle 45] ({-cos(45)+0.92},{sin(45)}) -- ({-0.7*cos(45)+0.92},{0.7*sin(45)});
\draw [dotted,->,-triangle 45] ({-cos(65)+0.92},{sin(65)}) -- ({-0.7*cos(65)+0.92},{0.7*sin(65)});
\draw [dotted,->,-triangle 45] ({-cos(85)+0.92},{sin(85)}) -- ({-0.7*cos(85)+0.92},{0.7*sin(85)});
\draw [dotted,->,-triangle 45] ({-cos(105)+0.92},{sin(105)}) -- ({-0.7*cos(105)+0.92},{0.7*sin(105)});
\draw [dotted,->,-triangle 45] ({-cos(140)+0.92},{sin(140)}) -- ({-0.7*cos(140)+0.92},{0.7*sin(140)});
\draw [dotted,->,-triangle 45] ({-cos(200)+0.92},{sin(200)}) -- ({-0.7*cos(200)+0.92},{0.7*sin(200)});
\draw [dotted,->,-triangle 45] ({-cos(220)+0.92},{sin(220)}) -- ({-0.7*cos(220)+0.92},{0.7*sin(220)});
\draw [dotted,->,-triangle 45] ({-cos(250)+0.92},{sin(250)}) -- ({-0.7*cos(250)+0.92},{0.7*sin(250)});
\draw [dotted,->,-triangle 45] ({-cos(270)+0.92},{sin(270)}) -- ({-0.7*cos(270)+0.92},{0.7*sin(270)});
\draw [dotted,->,-triangle 45] ({-cos(290)+0.92},{sin(290)}) -- ({-0.7*cos(290)+0.92},{0.7*sin(290)});
\path (0.08,-0.15) node {\footnotesize{$\emptyset$}}
(0.35,-0.7) node {\footnotesize{$1$}}
(-0.35,-0.7) node {\footnotesize{$3$}}
(0.08,0.38) node {\footnotesize{$2$}}
(0.14,-1.1) node {\footnotesize{$13$}}
(0,-1.6) node {\footnotesize{$132$}}
(-0.5,-1.5) node [gray] {\footnotesize{$121$}}
(+0.5,-1.5) node [gray] {\footnotesize{$323$}}
;
\path ({cos(45)-0.92+0.12},{sin(45)+0.06}) node {\footnotesize{$23$}}
({cos(65)-0.92+0.14},{sin(65)+0.06}) node {\footnotesize{$232$}}
({cos(85)-0.92},{sin(85)+0.1}) node {\footnotesize{$2323$}}
({cos(105)-0.92-0.05},{sin(105)+0.12}) node {\footnotesize{$23232$}}
({cos(140)-0.92-0.1},{sin(140)+0.08}) node {\footnotesize{$y23$}}
({cos(160)-0.92-0.1},{sin(160)+0.05}) node {\footnotesize{$y2$}}
({cos(180)-0.92+0.1},{sin(180)}) node {\footnotesize{$y$}}
({cos(200)-0.92+0.15},{sin(200)}) node {\footnotesize{$y3$}}
({cos(220)-0.92-0.15},{sin(220)}) node {\footnotesize{$y32$}}
({cos(255)-0.92-0.15},{sin(255)-0.1}) node {\footnotesize{$3232$}}
({cos(275)-0.92-0.08},{sin(275)-0.1}) node {\footnotesize{$323$}}
({cos(295)-0.92-0.08},{sin(295)-0.1}) node {\footnotesize{$32$}}
({cos(180)-0.92-0.325},{sin(180)+0.1}) node {\footnotesize{$y1$}}
({cos(180)-0.92-2*0.325},{sin(180)-0.1}) node {\footnotesize{$y12$}}
({cos(180)-0.92-0.325},{sin(200)-0.1}) node {\footnotesize{$y31$}}
;
\path [gray] 
({0.5*cos(45)-0.92+0.1},{0.5*sin(45)+0.06}) node {\footnotesize{$13$}}
({0.6*cos(65)-0.92},{0.6*sin(65)}) node {\footnotesize{$21$}}
({0.6*cos(85)-0.92},{0.6*sin(85)}) node {\footnotesize{$13$}}
({0.6*cos(105)-0.92},{0.6*sin(105)}) node {\footnotesize{$21$}}
({0.6*cos(140)-0.92},{0.6*sin(140)}) node {\footnotesize{$21$}}
({0.6*cos(160)-0.92},{0.6*sin(160)}) node {\footnotesize{$13$}}
({0.6*cos(220)-0.92},{0.6*sin(220)}) node {\footnotesize{$13$}}
({0.6*cos(255)-0.92-0.05},{0.6*sin(255)}) node {\footnotesize{$21$}}
({0.6*cos(275)-0.92-0.05},{0.6*sin(275)}) node {\footnotesize{$13$}}
({0.6*cos(295)-0.92-0.05},{0.6*sin(295)}) node {\footnotesize{$21$}}
({cos(180)-0.92-2*0.325-0.08},{sin(160)+0.08}) node {\footnotesize{$2121$}}
({cos(180)-0.92-2*0.325},{sin(200)-0.1}) node {\footnotesize{$212$}}
({cos(180)-0.7-4*0.325+0.1},{sin(180)+0.1}) node {\footnotesize{$323$}}
;
\path ({-cos(45)+0.92-0.12},{sin(45)+0.06}) node {\footnotesize{$21$}}
({-cos(65)+0.92-0.14},{sin(65)+0.06}) node {\footnotesize{$212$}}
({-cos(85)+0.92},{sin(85)+0.1}) node {\footnotesize{$2121$}}
({-cos(105)+0.92+0.05},{sin(105)+0.12}) node {\footnotesize{$21212$}}
({-cos(140)+0.92+0.1},{sin(140)+0.08}) node {\footnotesize{$x12$}}
({-cos(160)+0.92-0.12},{sin(160)}) node {\footnotesize{$x1$}}
({-cos(180)+0.92-0.1},{sin(180)}) node {\footnotesize{$x$}}
({-cos(200)+0.92+0.13},{sin(200)}) node {\footnotesize{$x2$}}
({-cos(220)+0.92+0.15},{sin(220)}) node {\footnotesize{$x21$}}
({-cos(255)+0.92+0.15},{sin(255)-0.1}) node {\footnotesize{$1212$}}
({-cos(275)+0.92+0.08},{sin(275)-0.1}) node {\footnotesize{$121$}}
({-cos(295)+0.92+0.08},{sin(295)-0.1}) node {\footnotesize{$12$}}
({-cos(180)+0.92+0.325},{sin(180)-0.1}) node {\footnotesize{$x3$}}
({-cos(180)+0.92+2*0.325},{sin(180)+0.1}) node {\footnotesize{$x32$}}
({-cos(180)+0.92+0.325},{sin(160)+0.1}) node {\footnotesize{$x13$}}
;
\path [gray] ({-0.5*cos(45)+0.92-0.1},{0.5*sin(45)+0.06}) node {\footnotesize{$13$}}
({-0.6*cos(65)+0.92},{0.6*sin(65)}) node {\footnotesize{$23$}}
({-0.6*cos(85)+0.92},{0.6*sin(85)}) node {\footnotesize{$13$}}
({-0.6*cos(105)+0.92},{0.6*sin(105)}) node {\footnotesize{$23$}}
({-0.6*cos(140)+0.92},{0.6*sin(140)}) node {\footnotesize{$13$}}
({-0.6*cos(200)+0.92},{0.6*sin(200)}) node {\footnotesize{$13$}}
({-0.6*cos(220)+0.92},{0.6*sin(220)}) node {\footnotesize{$23$}}
({-0.6*cos(255)+0.92+0.05},{0.6*sin(255)}) node {\footnotesize{$23$}}
({-0.6*cos(275)+0.92+0.05},{0.6*sin(275)}) node {\footnotesize{$13$}}
({-0.6*cos(295)+0.92+0.05},{0.6*sin(295)}) node {\footnotesize{$23$}}
({-cos(180)+0.92+2*0.325},{sin(160)+0.1}) node {\footnotesize{$232$}}
({-cos(180)+0.92+2*0.325+0.05},{sin(200)-0.1}) node {\footnotesize{$2323$}}
({-cos(180)+0.7+4*0.325-0.1},{sin(180)+0.1}) node {\footnotesize{$121$}}
;\end{tikzpicture}};
\end{tikzpicture}
\caption{Normal form automaton for Class~II with $a,b>4$ ($a$ even and $b$ odd in this example)}\label{fig:caseII}
\end{figure}

\begin{figure}[!h]
\centering
\begin{tikzpicture} [scale=4]
\node {\begin{tikzpicture} [scale=3.3]
\draw [dashed,->,domain=180:165,-triangle 45] plot ({cos(\x)},{sin(\x)});
\draw [->,domain=165:150,-triangle 45] plot ({cos(\x)},{sin(\x)});
\draw [dashed,->,domain=150:135,-triangle 45] plot ({cos(\x)},{sin(\x)});
\draw [->,domain=135:120,-triangle 45] plot ({cos(\x)},{sin(\x)});
\draw [dashed,->,domain=120:105,-triangle 45] plot ({cos(\x)},{sin(\x)});
\draw [->,domain=180:195,-triangle 45] plot ({cos(\x)},{sin(\x)});
\draw [dashed,->,domain=195:210,-triangle 45] plot ({cos(\x)},{sin(\x)});
\draw [->,domain=210:225,-triangle 45] plot ({cos(\x)},{sin(\x)});
\draw [dashed,->,domain=225:240,-triangle 45] plot ({cos(\x)},{sin(\x)});
\draw [->,domain=240:255,-triangle 45] plot ({cos(\x)},{sin(\x)});
\draw [dashed,->,domain=30:15,-triangle 45] plot ({cos(\x)},{sin(\x)});
\draw [->,domain=45:30,-triangle 45] plot ({cos(\x)},{sin(\x)});
\draw [dashed,->,domain=60:45,-triangle 45] plot ({cos(\x)},{sin(\x)});
\draw [->,domain=75:60,-triangle 45] plot ({cos(\x)},{sin(\x)});
\draw [dashed,->,domain=345:360,-triangle 45] plot ({cos(\x)},{sin(\x)});
\draw [->,domain=330:345,-triangle 45] plot ({cos(\x)},{sin(\x)});
\draw [dashed,->,domain=315:330,-triangle 45] plot ({cos(\x)},{sin(\x)});
\draw [->,domain=300:315,-triangle 45] plot ({cos(\x)},{sin(\x)});
\draw [dashed,->,domain=285:300,-triangle 45] plot ({cos(\x)},{sin(\x)});
\draw [->,gray,dashed,domain=105:75,-triangle 45] plot ({cos(\x)},{sin(\x)});
\draw [->,gray,dashed,domain=255:285,-triangle 45] plot ({cos(\x)},{sin(\x)});
\draw [dotted,->,-triangle 45] (-1,0) -- (-1.2,0);
\draw [dotted,->,-triangle 45] ({cos(195)},{sin(195)}) -- (-1.2,{sin(195)});
\draw [dotted,->,-triangle 45] ({cos(165)},{sin(165)}) -- (-1.2,{sin(165)});
\draw [dashed,->,-triangle 45] (-1.2,0) -- (-1.45,0);
\draw [dashed,->,-triangle 45] (-1.2,{sin(165)}) -- (-1.45,{sin(165)});
\draw [->,-triangle 45] (-1.45,{sin(165)}) -- (-1.7,{sin(165)});
\draw [dashed,->,-triangle 45] (-1.7,{sin(165)}) -- (-1.95,{sin(165)});
\draw [dotted,->,-triangle 45] (-1.95,{sin(165)}) -- (-1.95,{sin(165)+0.25});
\draw [->,-triangle 45] (-1.95,{sin(165)}) -- (-2.2,{sin(165)});
\draw [dotted,->,-triangle 45] (-2.2,{sin(165)}) -- (-2.2,{sin(165)+0.25});
\draw [dashed,->,-triangle 45] (-2.2,{sin(165)}) -- (-2.2,{sin(165)-0.25});
\draw [->,-triangle 45] (-1.45,{sin(180)}) -- (-1.7,{sin(180)});
\draw [dashed,->,-triangle 45] (-1.7,{sin(180)}) -- (-1.95,{sin(180)});
\draw [dashed,->,-triangle 45] (-1.2,{sin(195)}) -- (-1.45,{sin(195)});
\draw [->,-triangle 45] (-1.45,{sin(195)}) -- (-1.7,{sin(195)});
\draw [dashed,->,-triangle 45] (-1.7,{sin(195)}) -- (-1.95,{sin(195)});
\draw [dotted,->,-triangle 45] (-1.45,{sin(195)}) -- (-1.45,{sin(195)-0.25});
\draw [dotted,->,-triangle 45] (-1.7,{sin(195)}) -- (-1.7,{sin(195)-0.25});
\draw [dotted,->,-triangle 45] (1,{sin(0)}) -- (1.2,{sin(0)});
\draw [dotted,->,-triangle 45] ({cos(15)},{sin(15)}) -- (1.2,{sin(15)});
\draw [dotted,->,-triangle 45] ({cos(-15)},{sin(-15)}) -- (1.2,{sin(-15)});
\draw [dotted,->,-triangle 45] ({cos(-30)},{sin(-30)}) -- (1.2,{sin(-30)});
\draw [dashed,->,-triangle 45] (1.2,{sin(0)}) -- (1.45,{sin(0)});
\draw [dashed,->,-triangle 45] (1.2,{sin(-15)}) -- (1.45,{sin(-15)});
\draw [->,-triangle 45] (1.45,{sin(0)}) -- (1.7,{sin(0)});
\draw [->,-triangle 45] (1.45,{sin(-15)}) -- (1.7,{sin(-15)});
\draw [dashed,->,-triangle 45] (1.7,{sin(0)}) -- (1.95,{sin(0)});
\draw [dashed,->,-triangle 45] (1.7,{sin(-15)}) -- (1.95,{sin(-15)});
\draw [dotted,->,-triangle 45] (1.95,{sin(0)}) -- (2.2,{sin(0)});
\draw [->,-triangle 45] (1.95,{sin(0)}) -- (1.95,{sin(15)});
\draw [dotted,->,-triangle 45] (1.95,{sin(15)}) -- (2.2,{sin(15)});
\draw [dashed,->,-triangle 45] (1.95,{sin(15)}) -- (1.7,{sin(15)});
\draw [dashed,->,-triangle 45] (1.2,{sin(15)}) -- (1.45,{sin(15)});
\draw [dashed,->,-triangle 45] (1.2,{sin(-30)}) -- (1.45,{sin(-30)});
\draw [->,-triangle 45] (-1.2,{sin(180-15)}) -- (-1.2,{sin(180-15)+0.25});
\draw [->,dotted,-triangle 45] ({cos(30)},{sin(30)}) -- ({0.8*cos(30)},{0.8*sin(30)});
\draw [->,dotted,-triangle 45] ({cos(45)},{sin(45)}) -- ({0.8*cos(45)},{0.8*sin(45)});
\draw [->,dotted,-triangle 45] ({cos(60)},{sin(60)}) -- ({0.8*cos(60)},{0.8*sin(60)});
\draw [->,dotted,-triangle 45] ({cos(75)},{sin(75)}) -- ({0.8*cos(75)},{0.8*sin(75)});
\draw [->,dotted,-triangle 45] ({cos(180-2*15)},{sin(180-2*15)}) -- ({0.8*cos(180-2*15)},{0.8*sin(180-2*15)});
\draw [->,dotted,-triangle 45] ({cos(180-3*15)},{sin(180-3*15)}) -- ({0.8*cos(180-3*15)},{0.8*sin(180-3*15)});
\draw [->,dotted,-triangle 45] ({cos(180-4*15)},{sin(180-4*15)}) -- ({0.8*cos(180-4*15)},{0.8*sin(180-4*15)});
\draw [->,dotted,-triangle 45] ({cos(180-5*15)},{sin(180-5*15)}) -- ({0.8*cos(180-5*15)},{0.8*sin(180-5*15)});
\draw [->,dotted,-triangle 45] ({cos(180+2*15)},{sin(180+2*15)}) -- ({0.8*cos(180+2*15)},{0.8*sin(180+2*15)});
\draw [->,dotted,-triangle 45] ({cos(180+3*15)},{sin(180+3*15)}) -- ({0.8*cos(180+3*15)},{0.8*sin(180+3*15)});
\draw [->,dotted,-triangle 45] ({cos(180+4*15)},{sin(180+4*15)}) -- ({0.8*cos(180+4*15)},{0.8*sin(180+4*15)});
\draw [->,dotted,-triangle 45] ({cos(180+5*15)},{sin(180+5*15)}) -- ({0.8*cos(180+5*15)},{0.8*sin(180+5*15)});
\draw [->,dotted,-triangle 45] ({cos(-45)},{sin(-45)}) -- ({0.8*cos(-45)},{0.8*sin(-45)});
\draw [->,dotted,-triangle 45] ({cos(-60)},{sin(-60)}) -- ({0.8*cos(-60)},{0.8*sin(-60)});
\draw [->,dotted,-triangle 45] ({cos(-75)},{sin(-75)}) -- ({0.8*cos(-75)},{0.8*sin(-75)});
\path (-0.95,0) node {\footnotesize{$\emptyset$}}
({cos(180+15)+0.07},{sin(180+15)}) node {\footnotesize{$1$}}
({cos(180-15)+0.07},{sin(180-15)}) node {\footnotesize{$2$}}
(-1.2,{sin(180-15)-0.07}) node {\footnotesize{$23$}}
(-1.2,{sin(180-15)+0.3}) node [gray] {\footnotesize{$13$}}
(-1.2,{sin(180)+0.07}) node {\footnotesize{$3$}}
(-1.2,{sin(180)-0.25-0.07}) node {\footnotesize{$13$}}
(-1.45,{sin(180)-0.25+0.07}) node {\footnotesize{$132$}}
(-1.7,{sin(180)-0.25+0.07}) node {\footnotesize{$1321$}}
(-1.7,{sin(180)-0.55}) node [gray] {\footnotesize{$2321$}}
(-1.45,{sin(180)-0.55}) node [gray] {\footnotesize{$232$}}
(-2.1,{sin(180)-0.257}) node [gray] {\footnotesize{$1212$}}
(-1.45,{sin(180)-0.07}) node {\footnotesize{$32$}}
(-1.7,{sin(180)-0.07}) node {\footnotesize{$321$}}
(-1.95,{sin(180)+0.075}) node [gray] {\footnotesize{$212$}}
(-1.45,{sin(180-15)+0.07}) node {\footnotesize{$232$}}
(-1.7,{sin(180-15)+0.07}) node {\footnotesize{$2321$}}
(-1.95,{sin(180-15)-0.07}) node {\footnotesize{$23212$}}
(-2.35,{sin(180-15)}) node {\footnotesize{$232121$}}
(-2.2,{sin(180-15)+0.3}) node [gray] {\footnotesize{$2321$}}
(-2.22,{sin(180-15)-0.3}) node [gray] {\footnotesize{$21212$}}
(-1.95,{sin(180-15)+0.3}) node [gray] {\footnotesize{$232$}}
({cos(180-30)-0.09},{sin(180-30)+0.02}) node {\footnotesize{$21$}}
({cos(180-45)-0.11},{sin(180-45)+0.02}) node {\footnotesize{$212$}}
({cos(180-60)-0.11},{sin(180-60)+0.06}) node {\footnotesize{$2121$}}
({cos(180-75)-0.09},{sin(180-75)+0.08}) node {\footnotesize{$21212$}}
({cos(180+30)-0.09},{sin(180+30)-0.02}) node {\footnotesize{$12$}}
({cos(180+45)-0.14},{sin(180+45)-0.02}) node {\footnotesize{$121$}}
({cos(180+60)-0.11},{sin(180+60)-0.06}) node {\footnotesize{$1212$}}
({cos(180+75)-0.09},{sin(180+75)-0.08}) node {\footnotesize{$12121$}}
({0.8*cos(180-30)+0.07},{0.8*sin(180-30)}) node [gray] {\footnotesize{$13$}}
({-0.8*cos(180-30)-0.07},{0.8*sin(180-30)}) node [gray] {\footnotesize{$13$}}
({0.8*cos(180-45)+0.05},{0.8*sin(180-45)-0.03}) node [gray] {\footnotesize{$23$}}
({-0.8*cos(180-45)-0.05},{0.8*sin(180-45)-0.03}) node [gray] {\footnotesize{$23$}}
({0.8*cos(180-60)+0.02},{0.8*sin(180-60)-0.06}) node [gray] {\footnotesize{$13$}}
({-0.8*cos(180-60)-0.02},{0.8*sin(180-60)-0.06}) node [gray] {\footnotesize{$13$}}
({0.8*cos(180-75)+0.02},{0.8*sin(180-75)-0.05}) node [gray] {\footnotesize{$23$}}
({-0.8*cos(180-75)-0.02},{0.8*sin(180-75)-0.05}) node [gray] {\footnotesize{$23$}}
({0.8*cos(180+30)+0.07},{0.8*sin(180+30)}) node [gray] {\footnotesize{$23$}}
({0.8*cos(180+45)+0.05},{0.8*sin(180+45)+0.03}) node [gray] {\footnotesize{$13$}}
({0.8*cos(180+60)+0.02},{0.8*sin(180+60)+0.06}) node [gray] {\footnotesize{$23$}}
({0.8*cos(180+75)+0.02},{0.8*sin(180+75)+0.05}) node [gray] {\footnotesize{$13$}}
({-0.8*cos(180+75)-0.02},{0.8*sin(180+75)+0.05}) node [gray] {\footnotesize{$13$}}
({-0.8*cos(180+60)-0.02},{0.8*sin(180+60)+0.06}) node [gray] {\footnotesize{$23$}}
({-0.8*cos(180+45)-0.05},{0.8*sin(180+45)+0.03}) node [gray] {\footnotesize{$13$}}
({cos(0)-0.08},{sin(0)}) node {\footnotesize{$x$}}
({1.2},{-0.05}) node {\footnotesize{$x3$}}
({1.2},{-0.175}) node {\footnotesize{$x23$}}
({1.2},{-0.57}) node {\footnotesize{$x213$}}
({1.52},{-0.5}) node [gray] {\footnotesize{$232$}}
({1.45},{0.075}) node {\footnotesize{$x32$}}
({1.45},{-0.33}) node {\footnotesize{$x232$}}
({1.7},{-0.33}) node {\footnotesize{$x2321$}}
({1.95},{-0.19}) node [gray] {\footnotesize{$1212$}}
({2.2},{0.07}) node [gray] {\footnotesize{$232$}}
({2.2},{sin(15)-0.07}) node [gray] {\footnotesize{$2321$}}
({1.7},{0.075}) node {\footnotesize{$x321$}}
({1.76},{sin(15)-0.08}) node [gray] {\footnotesize{$121212$}}
({1.95},{-0.065}) node {\footnotesize{$x3212$}}
({1.95},{0.32}) node {\footnotesize{$x32121$}}
({cos(15)-0.1},{sin(15)}) node {\footnotesize{$x1$}}
({1.2},{sin(15)+0.075}) node {\footnotesize{$x13$}}
({1.4},{sin(15)-0.08}) node [gray] {\footnotesize{$232$}}
({cos(30)+0.12},{sin(30)}) node {\footnotesize{$x12$}}
({cos(45)+0.14},{sin(45)+0.02}) node {\footnotesize{$x121$}}
({cos(60)+0.14},{sin(60)+0.06}) node {\footnotesize{$x1212$}}
({cos(75)+0.13},{sin(75)+0.07}) node {\footnotesize{$x12121$}}
({cos(-15)-0.1},{sin(-15)}) node {\footnotesize{$x2$}}
({cos(-30)-0.125},{sin(-30)}) node {\footnotesize{$x21$}}
({cos(-45)+0.125},{sin(-45)}) node {\footnotesize{$x212$}}
({cos(-60)+0.1},{sin(-60)-0.05}) node {\footnotesize{$x2121$}}
({cos(-75)+0.13},{sin(-75)-0.07}) node {\footnotesize{$x21212$}}
;\end{tikzpicture}};
\end{tikzpicture}
\caption{Normal form automaton for Class~III (with $a$ even in this example)}\label{fig:caseIII}
\end{figure}

Finally, for a Class IV Fuchsian Coxeter system with $S=\{s_1,\ldots,s_n\}$ the normal form automaton is constructed from the Cannon automaton by removing the arrows
$$
\underbrace{s_{i+1}s_is_{i+1}\cdots}_{\text{$m_{i,i+1}-1$ terms}}\rightarrow \underbrace{s_{i+1}s_is_{i+1}\cdots}_{\text{$m_{i,i+1}$ terms}}\quad\text{and}\quad \underbrace{s_{n}s_1s_{n}\cdots}_{\text{$m_{1,n}-1$ terms}}\rightarrow \underbrace{s_{n}s_1s_{n}\cdots}_{\text{$m_{1,n}$ terms}}
$$
for $i=1,\ldots,n-1$.

We now show that the normal form automaton for each class is strongly connected. In the notation of Figure~\ref{fig:caseI} we will simply write $w$ for the cone type $T(w)$. 

Consider Class~I first. Note that that the cone types $\emptyset$, $1$, $2$ and $3$ are obviously not recurrent. Since $m_{12}=a>3$ there is a cycle in the normal form automaton
$$
c=(12\to23\to31\to12\to 121\to13\to32\to 21\to 212\to 23\to31\to 12)
$$
containing all cone types $w$ with $\ell(w)=2$ (it is important here that $121$ and $212$ are not the longest words of $W_{12}$, since $m_{12}>3$). Next we claim that for each cone type $w$ there is a path $\gamma_1$ in the normal form automaton to some cone of the form $ij$ (with $i\neq j$). This is immediate if $w\notin\{x,y,z,x3,y1,z2\}$, and for these remaining cases we note that (since $m_{12}>3$):
\begin{align*}
x&\to x3\to 131\to 12&&\text{if $m_{13}>3$}\\
x&\to x3\to 131=z\to z2\to 121\to13&&\text{if $m_{13}=3$}\\
y&\to y1\to 212\to 23\\
z&\to z2\to 121 \to 13 
\end{align*}
Next, it is clear that for all cone types $w$ other than $\emptyset,1,2,3$ there is a path $\gamma_2$ in the automaton from a cone type of the form $ij$ to $w$. Thus, using $c,\gamma_1$, and $\gamma_2$, we see that for each cone type other than $\emptyset,1,2,3$ there is a loop $\gamma$ from $12$ to $12$ passing through $w$. This shows that each cone type other than $\emptyset,1,2,3$ is recurrent, and that the normal form automaton is strongly connected.

Now consider Class~II, and for simplicity consider the case of Figure~\ref{fig:caseII}, that is, $m_{12},m_{23}>4$ with $a=m_{12}$ even and $b=m_{23}$ odd (the other cases are similar). It is clear that the cone types $\emptyset,1,2,3,12,32$ are not recurrent, and we claim that all other cone types are recurrent, and that the Cannon automaton is strongly connected. To see this, consider the paths:
\begin{align*}
\gamma_1&=(21\to212\to\cdots\to x1\to x13\to 232\to2323\to 13\to 132\to 121)\\
\gamma_2&=(121\to 1212\to\cdots\to x\to x3\to x32\to 121)\\
\gamma_3&=(121\to 23\to 232\to\cdots\to y\to y1\to y12 \to 323)\\
\gamma_4&=(323\to 3232\to\cdots\to y3\to y31\to 212)\\
\gamma_5&=(212\to 23\to 13\to 132\to 323\to 3232\to 21).
\end{align*}
The concatenation $\gamma=\gamma_1\gamma_2\gamma_3\gamma_4\gamma_5$ is a loop visiting all cone types other than $\emptyset,1,2,3,12,32$, hence the result. 

Now consider Class~III, and for simplicity consider the case of Figure~\ref{fig:caseIII}, that is, $a=m_{12}>6$ even (the case $a$ odd is similar). It is clear that the cone types $\emptyset,1,2,3,12,21,32,121,212,321$, $2121$ are not recurrent. We claim that all other cone types are recurrent, and that the automaton is strongly connected. Define the following paths:
\begin{align*}
\gamma_1&=(232\to 2321\to 23212\to 232121\to 21212\to \cdots\to x1\to x13\to 232)\\
\gamma_2&=(232\to 2321\to 23212\to 232121\to 21212\to 23\to 13\to 132\to 1321\to 1212)\\
\gamma_3&=(1212\to \cdots\to x21\to x213\to 232)\\
\gamma_4&=(1212\to\cdots\to x2\to x23\to x232\to x2321\to 1212)\\
\gamma_5&=(1212\to \cdots\to x\to x3\to x32\to x321\to x3212\to x32121\to 121212)\\
\gamma_6&=(121212\to\cdots \to x\to x3\to x32\to x321\to x3212\to 232).
\end{align*} 
The concatenation $\gamma=\gamma_1\gamma_2\gamma_3\gamma_2\gamma_4\gamma_5\gamma_6$ is a loop starting and finishing at $232$ and including every cone type other than $\emptyset,1,2,3,12,21,32,121,212,321,2121$, hence the result.

Finally, consider the groups in Class~IV. Let $1,2,\ldots,n$ be the generators of $W$, arranged cyclically around the fundamental chamber. If $n\geq 5$, then for each pair $(i,i+1)$ there is a generator $j$ with $m_{i,j}=\infty$ and $m_{{i+1},j}=\infty$, and thus $i\to j \to {i+1}$ in the automaton (see \cite[Lemma~A.6]{gilch-mueller-parkinson:14}). Moreover, for any $w\in W_{i,i+1}$ we have $w\to j$, and it follows that every node other than $\emptyset$ is recurrent, and moreover that the normal form automaton is strongly connected. If $n=4$ then we may assume that $m_{12}\geq 3$ (for if $m_{ij}=2$ for all $i,j$ then $W$ is affine type $\tilde{A}_1\times\tilde{A}_1$, where $\tilde{A}_1$ is the infinite dihedral group). Then $21\to 3$ and $12\to 4$. Thus $1\to 12\to 4\to 2\to21\to3\to1$. Again it follows that every node other than $\emptyset$ is recurrent, and that the normal form automaton is strongly connected.
\end{proof}

\end{appendix}

\bibliographystyle{abbrv}
\bibliography{literatur}
\medskip

\noindent\begin{minipage}{0.5\textwidth}
Lorenz Gilch\newline
Institut f\"{u}r Mathematische Strukturtheorie\newline
Technische Universit\"{a}t Graz\newline
Steyrergasse 30\newline
8010 Graz, Austria\newline
\texttt{Lorenz.Gilch@freenet.de}
\end{minipage}
\begin{minipage}{0.5\textwidth}
\noindent Sebastian M\"{u}ller\newline
Aix Marseille Universit\'e\newline
CNRS  Centrale Marseille\newline 
Institut de Math\'ematiques de Marseille (I2M)\newline
UMR 7373\newline 13453 Marseille France\newline
\texttt{sebastian.muller@univ-amu.fr}
\end{minipage}
\bigskip

\noindent James Parkinson\newline
School of Mathematics and Statistics\newline
University of Sydney\newline
Carslaw Building, F07\newline
NSW, 2006, Australia\newline
\texttt{jamesp@maths.usyd.edu.au}

\end{document}